\documentclass[12pt]{amsart}
\usepackage
{
	amssymb,
	amsmath,
	caption,
	color,
	comment,
	dashbox,
	hyperref,
	mathrsfs,
	stmaryrd,
	tabularx,
	textcomp,
	tikz,
	tikz-cd,
	todonotes,
	verbatim
}
\usepackage[all,pdf]{xy}
\usepackage[noadjust]{cite}
\usepackage[shortlabels]{enumitem}
\usepackage[abs]{overpic}
\usepackage[margin=1in]{geometry}

\usetikzlibrary{calc}
\usetikzlibrary{arrows,decorations.markings}

\newcolumntype{L}{>{\raggedright\arraybackslash}m{.3042\textwidth}}
\newcolumntype{H}{>{\centering\arraybackslash}m{.3042\textwidth}}
\newcolumntype{M}{>{\raggedright\arraybackslash}m{.4674\textwidth}}
\newcolumntype{W}{>{\centering\arraybackslash}m{.4674\linewidth}}

\newcommand{\pic}{\includegraphics[scale=.6]}
\graphicspath{{appendix_portraits}}

\newtheorem{thm}{Theorem}[section]
\newtheorem{lem}[thm]{Lemma}
\newtheorem{prop}[thm]{Proposition}
\newtheorem{cor}[thm]{Corollary}
\newtheorem{conj}[thm]{Conjecture}

\theoremstyle{definition}
\newtheorem{defn}[thm]{Definition}

\newtheorem{ex}[thm]{Example}
\newtheorem{ques}[thm]{Question}
\newtheorem{rem}[thm]{Remark}

\newcommand{\bbA}{\mathbb{A}}

\newcommand{\bbF}{\mathbb{F}}

\newcommand{\bbP}{\mathbb{P}}
\newcommand{\bbQ}{\mathbb{Q}}

\newcommand{\bbZ}{\mathbb{Z}}

\newcommand{\calD}{\mathcal{D}}

\newcommand{\calO}{\mathcal{O}}
\newcommand{\calP}{\mathcal{P}}

\newcommand{\calS}{\mathcal{S}}
\newcommand{\calT}{\mathcal{T}}

\newcommand{\frakc}{\mathfrak{c}}

\newcommand{\frakp}{\mathfrak{p}}

\newcommand{\Kbar}{\overline{K}}

\newcommand{\Pbar}{\overline{P}}

\renewcommand{\hbar}{\overline{h}}

\newcommand{\QQbar}{\overline{\bbQ}}

\newcommand{\disc}{\operatorname{disc}}

\newcommand{\Ell}{{\operatorname{ell}}}

\newcommand{\Gal}{\operatorname{Gal}}
\newcommand{\gon}{\operatorname{gon}}

\newcommand{\ord}{\operatorname{ord}}

\newcommand{\PGL}{\operatorname{PGL}}

\newcommand{\Spec}{\operatorname{Spec}}
\newcommand{\sqf}{\operatorname{sqf}}

\newcommand{\rat}{{\mathrm{rat}}}
\newcommand{\Quad}{{\mathrm{quad}}}

\renewcommand{\bar}{\overline}

\renewcommand{\tilde}{\widetilde}
\renewcommand{\phi}{\varphi}

\newcommand{\longto}{\longrightarrow}

\newcommand{\V}{t}

\newcommand{\tfae}{(\roman*)}

\newcommand{\Mod}[1]{\ \left(\operatorname{mod}\ {#1}\right)}

\makeatletter
\let\@@pmod\pmod
\DeclareRobustCommand{\pmod}{\@ifstar\@pmods\@@pmod}
\def\@pmods#1{\mkern4mu({\operator@font mod}\mkern 6mu#1)}
\makeatother

\numberwithin{equation}{section}

\title{Quadratic points on dynamical modular curves}
\author{John R. Doyle}
\address{Department of Mathematics, Oklahoma State University,
Stillwater, OK 74078} 
\email{john.r.doyle@okstate.edu}

\author{David Krumm}
\urladdr{http://maths.dk}
\email{david.krumm@gmail.com}

\date{\today}

\begin{document}
\begin{abstract}
Among all the dynamical modular curves associated to quadratic polynomial maps, we determine which curves have infinitely many quadratic points. This yields a classification statement on preperiodic points for quadratic polynomials over quadratic fields, extending previous work of Poonen, Faber, and the authors.
\end{abstract}

\maketitle

\section{Introduction}\label{sec:intro}
Let $K$ be a field with algebraic closure $\bar K$, and let $f$ be a rational function in one variable over $K$.
Corresponding to $f$ there are a morphism of algebraic varieties $\bbP_K^1\to\bbP_K^1$ and a map on point sets $\bbP^1(\bar K)\to\bbP^1(\bar K)$, both of which we also denote by $f$. A point $P\in\bbP^1(\bar K)$ is called {\bf periodic} for $f$ if there exists a positive integer $n$ such that $f^n(P)=P$, where $f^n$ denotes the $n$-fold composition of $f$ with itself; in that case, the smallest such $n$ is called the ({\bf exact}) {\bf period} of $P$. More generally, the point $P$ is {\bf preperiodic} for $f$ if there exists $m\ge 0$ such that $f^m(P)$ is periodic; the smallest such $m$ is then called the {\bf preperiod} of $P$, and we call the period of $f^m(P)$ the {\bf eventual period} of $P$. Here, $f^0$ is interpreted as the identity map, so that periodic points are considered preperiodic.

For any intermediate field $K\subseteq L\subseteq\bar K$ we define a directed graph $G(f,L)$, called the {\bf preperiodic portrait} of $f$ over $L$, whose vertices are the points $P\in\bbP^1(L)$ that are preperiodic for $f$, and whose directed edges are the ordered pairs $(P,f(P))$ for all vertices $P$. In the terminology of graph theory, $G(f,L)$ is a \emph{functional graph}, i.e., a directed graph in which every vertex has out-degree 1. Throughout this paper we will use the term {\bf portrait} instead of functional graph in order to emphasize our dynamical perspective.
 
 \subsection{Portraits for quadratic maps} Assume henceforth that $K$ is a number field. We will primarily, though not exclusively, be interested in the case where $K$ is a quadratic extension of $\bbQ$; we refer to such fields simply as {\it quadratic fields}. A type of problem that has received much attention in the field of arithmetic dynamics is that of classifying the portraits $G(f,K)$ up to graph isomorphism as $f$ is allowed to vary in an infinite family of rational functions. An early example of this classification problem is Poonen's study \cite{poonen:1998} of the portraits $G(f,\bbQ)$ as $f$ varies over the family of all quadratic polynomials with rational coefficients.

 \begin{thm}[Poonen \cite{poonen:1998}]\label{poonen_quadratic}
     Assume that there is no quadratic polynomial over $\bbQ$ having a rational periodic point of period greater than $3$. 
     Then, for every quadratic polynomial $f \in \bbQ[z]$, the portrait $G(f,\bbQ)$ is isomorphic to one of the following twelve graphs (using the labels from Appendix \ref{graphs_appendix}):
     \[
     \emptyset,\; 2(1),\; 3(1,1),\; 3(2),\; 4(1,1),\; 4(2),\; \rm5(1,1)a,\; 6(1,1),\; 6(2),\; 6(3),\; 8(2,1,1),\; 8(3).
     \]
 \end{thm}

Regarding the assumption in Theorem~\ref{poonen_quadratic}, it is known that a quadratic polynomial over $\bbQ$ cannot have rational periodic points of period $4$ (Morton \cite{morton:1998}), period $5$ (Flynn--Poonen--Schaefer \cite{flynn/poonen/schaefer:1997}), or, assuming that the conclusions of the Birch and Swinnerton-Dyer conjecture hold for a certain Jacobian variety, period $6$ (Stoll \cite{stoll:2008}).
In addition, a substantial amount of empirical evidence supporting the assumption in Poonen's theorem has been provided by Hutz and Ingram \cite{hutz/ingram:2013} and Benedetto et al. \cite{benedetto/etal:2014}. However, it remains an open problem to prove that this assumption is valid. Portraits for other families of quadratic maps over $\bbQ$ are studied in the articles \cite{manes:2008,lukas/etal:2014, canci/etal:2015,canci/vishkautsan:2017}. The present paper concerns the preperiodic portraits of quadratic polynomials defined over quadratic fields, a topic previously explored in \cite{doyle/faber/krumm:2014, doyle:2018quad, doyle:2020}.
 
\subsection{Analogy with torsion points} A guiding principle that has proved fruitful in arithmetic dynamics is to regard the set of preperiodic points of a map as being analogous to the set of torsion points on an abelian variety. Thus, for instance, the well-known fact that the set of $K$-rational torsion points on an abelian variety is finite is viewed as analogous to a theorem of Northcott  \cite{northcott:1950} stating that, for every rational function $f$ over $K$ of degree at least $2$, the set of $K$-rational preperiodic points of $f$ is finite.

Motivated by this analogy, Morton and Silverman formulated the following dynamical analogue of a standard uniform boundedness conjecture for abelian varieties. We state the dynamical conjecture only in the case of endomorphisms of the projective line, although a similar statement applies to arbitrary projective spaces\footnote{Interestingly, work of Fakhruddin \cite{fakhruddin:2003} shows that the more general Morton--Silverman conjecture in fact implies its analogue for abelian varieties. See also \cite[\textsection 3.3]{silverman:2007}, where Merel's theorem for elliptic curves is shown to follow from Conjecture \ref{conj:morton-silverman}}.

 \begin{conj}[Morton--Silverman \cite{morton/silverman:1994}]\label{conj:morton-silverman}
For a number field $K$ and morphism ${f:\bbP_K^1\to\bbP_K^1}$ of degree greater than $1$, the number of $K$-rational preperiodic points of $f$ is bounded above by a constant depending only on the degree of $f$ and the absolute degree of $K$.
 \end{conj}
 
  This dynamical uniform boundedness conjecture would imply, in particular, that there are only finitely many isomorphism classes of portraits $G(f,\bbQ)$
 as $f$ ranges over all quadratic polynomials
 with rational coefficients, since the number of vertices in such a portrait is uniformly bounded. Theorem \ref{poonen_quadratic} can thus be seen as a refinement of the conjecture in this case, as it provides a (conditionally) complete list of all possible portraits for the family of quadratic polynomial maps. In the analogy with torsion points, Poonen's list of portraits corresponds to the list of abelian groups that can be realized as the torsion subgroup of an elliptic curve over $\bbQ$, the latter list being provided by a well-known theorem of Mazur \cite{mazur:1977}. 
 
 Similarly, the Morton--Silverman conjecture would imply that the portraits $G(f,K)$, where $K$ is a quadratic field and $f$ is a quadratic polynomial over $K$, fall into finitely many isomorphism classes. A conjecturally complete list of classes was first proposed in \cite{doyle/faber/krumm:2014}, and is included here in Appendix~\ref{graphs_appendix}. The list is comprised of 46 portraits, and can be viewed as analogous to the list of 26 abelian groups, known by work of Kamienny \cite{kamienny} and Kenku--Momose \cite{kenku/momose:1988}, that can arise as torsion subgroups of elliptic curves over quadratic fields. 
 
 \subsection{Infinitely occurring portraits}
 Our primary objective in this paper is to determine, under a suitable notion of equivalence of maps, which of the 46 graphs in \cite{doyle/faber/krumm:2014} arise
 as the preperiodic portrait of infinitely many inequivalent quadratic polynomials over quadratic fields.
 In the context of elliptic curves, both over $\bbQ$ and over quadratic fields, the corresponding question is well understood: \emph{every} abelian group that arises as the torsion subgroup of an elliptic curve can be realized as such by infinitely many non-isomorphic curves (see \cite{jeon/kim/schweizer:2004}). 
 
 In order to state our questions more precisely, we begin by defining the appropriate equivalence relation on maps. Two morphisms $f,h:\bbP^1_K\to\bbP^1_K$ are called {\bf linearly conjugate} over $K$ if there exists an automorphism $\sigma\in\PGL_2(K)$ such that \[h=\sigma^{-1}\circ f\circ\sigma.\]
(Similarly, one can define linear conjugacy over any extension of $K$.) In that case, a simple argument shows that the portraits $G(f,K)$ and $G(h,K)$ are isomorphic as directed graphs; hence, the isomorphism class of $G(f,K)$ is determined by the linear conjugacy class of $f$. 

In the case of quadratic polynomials, it is well known that every such map $f\in K[z]$ is linearly conjugate to a unique map of the form
\[f_c(z):=z^2+c,\] where $c\in K$. Thus, in studying the portraits of quadratic polynomials we may restrict attention to the one-parameter family of maps $f_c$. 
 
 Returning to the question of portraits arising infinitely often, the case of quadratic polynomials over $\bbQ$ was answered by Faber, who showed in addition that Poonen's list in \cite{poonen:1998} does not omit any such portrait.
 
 \begin{thm}[Faber \cite{faber:2015}]\label{xander_thm} 
For a portrait $\calP$, the following are equivalent:
    \begin{enumerate}[\tfae]
        \item There exist infinitely many $c \in \bbQ$ such that $G(f_c,\bbQ) \cong \calP$.
        \item $\calP$ is isomorphic to one of the following graphs (using the labels from Appendix~\ref{graphs_appendix}):
\[\emptyset,\;4(1,1),\;4(2),\;6(1,1),\;6(2),\; 6(3),\;8(2,1,1).\]
    \end{enumerate}
 \end{thm}

Motivated by Faber's theorem, we now state the main questions to be addressed here.

 \begin{ques}\label{intro_question}
 Among the $46$ known isomorphism classes of portraits arising as $G(f_c,K)$, with $K$ a quadratic field and $c\in K$, which ones can be realized as such by infinitely many algebraic numbers $c$? In addition, must every infinitely occurring portrait belong to one of the $46$ known isomorphism classes? 
 \end{ques}
 
\subsection{Main results} We define the following sets of portraits using labels as in Appendix~\ref{graphs_appendix}:
 \begin{align*}
     &\Gamma_0 :=\{\emptyset,\;4(1,1),\;4(2),\;6(1,1),\;6(2),\;6(3),\;8(2,1,1)\};\\
     &\Gamma_\rat := \{\rm8(1,1)a,\;8(2)a,\;8(4),\;10(3,1,1),\;10(3,2)\};\\
     &\Gamma_\Quad := \{\rm8(1,1)b,\;8(2)b,\;8(3),\;10(2,1,1)a/b\};\\
     &\Gamma := \Gamma_0\cup\Gamma_\rat\cup\Gamma_\Quad.
 \end{align*}
 
We provide two answers to Question~\ref{intro_question} which differ in their level of specificity. The simplest is Theorem~\ref{thm:main}, with Theorems~\ref{main_thm_rational} and \ref{main_thm_quadratic} providing additional information in terms of the above subsets of $\Gamma$. For an integer $n \ge 1$, we define
    \[
        \bbQ^{(n)} := \{\alpha \in \QQbar : [\bbQ(\alpha) : \bbQ] \le n\}.
    \]
 
\begin{thm}\label{thm:main}
For a portrait $\calP$, the following are equivalent:
    \begin{enumerate}[\tfae]
        \item There exist infinitely many $c \in \bbQ^{(2)}$ such that $G(f_c,K) \cong \calP$ for some quadratic field $K$ containing $c$.
        \item $\calP \in \Gamma$.
    \end{enumerate}
\end{thm} 

Theorem~\ref{thm:main} can be refined in order to take into account certain subtleties illustrated by the following example: We see from Theorem~\ref{xander_thm} that the portrait $\calP = \rm{4(2)}$ is realized as $G(f_c,\bbQ)$ for infinitely many $c \in \bbQ$. For each such $c$, the set of preperiodic points for $f_c$ is a set of bounded height, and therefore $f_c$ has only finitely many preperiodic points of algebraic degree $2$ over $\bbQ$. Hence, for each of the infinitely many $c \in \bbQ$ with $G(f_c,\bbQ) \cong \calP$, we must also have $G(f_c,K) \cong \calP$ for all but finitely many quadratic fields $K$.

We therefore show that each of the portraits $\calP \in \Gamma$ is realized infinitely often---even if one excludes the infinitely many ``trivial" realizations in the sense of the previous paragraph.
This is done in the next two theorems, which are stated separately in order to distinguish between polynomials with rational coefficients and those with quadratic algebraic coefficients.

\begin{thm}\label{main_thm_rational}
For a portrait $\calP$, the following are equivalent:
    \begin{enumerate}[\tfae]
        \item There exist infinitely many $c \in \bbQ$ such that $G(f_c,\bbQ) \subsetneq G(f_c,K) \cong \calP$ for some quadratic field $K$.
        \item $\calP \in \Gamma \smallsetminus \{\emptyset, {\rm 6(3)}\}$.
    \end{enumerate}
\end{thm}

\begin{thm}\label{main_thm_quadratic}
For a portrait $\calP$, the following are equivalent:
    \begin{enumerate}[\tfae]
        \item There exist infinitely many 
        $c \in \bbQ^{(2)} \smallsetminus \bbQ$
        such that $G(f_c, \bbQ(c)) \cong \calP$.
        \item $\calP \in \Gamma_0 \cup \Gamma_\Quad$.
    \end{enumerate}
\end{thm}

Note that every portrait in $\Gamma$ is covered by at least one of Theorems~\ref{main_thm_rational} and \ref{main_thm_quadratic}, since $\emptyset$ and $\rm 6(3)$ are elements of $\Gamma_0$. In particular, these two theorems together imply Theorem~\ref{thm:main}.
 
\subsection{Quadratic points on dynamical modular curves} 
The proofs of our main results rely heavily on the concept of a \emph{dynamical modular curve}. To each of the 46 portraits from \cite{doyle/faber/krumm:2014}, and more generally to any portrait that could potentially be realized as the preperiodic portrait of a quadratic polynomial over a number field, we associate an algebraic curve parametrizing instances of the portrait as a preperiodic portrait $G(f_c,K)$. The curve corresponding to a portrait $\calP$ will be denoted 
$X_1(\calP)$ by analogy with the classical modular curves $X_1(N)$ parametrizing elliptic curves with a torsion point of order $N$. To avoid confusion, the latter curve will henceforth be denoted $X_1^\Ell(N)$. The details of the construction as well as basic properties of dynamical modular curves are discussed in \cite{doyle:2019}. A more general construction of dynamical moduli spaces appears in \cite{doyle/silverman:2020}.

The core of our analysis in this paper is a study of basic geometric invariants, such as genus and gonality, of the curves $X_1(\calP)$. We then turn this geometric data into arithmetic data using Faltings' theorem on rational points on subvarieties of abelian varieties, via the following result of Harris and Silverman:

\begin{thm}[Harris--Silverman {\cite[Cor. 3]{harris/silverman:1991}}]\label{thm:harris/silverman}
Let $X$ be a smooth, irreducible, projective curve of genus $g \ge 2$ defined over a number field $K$. If $X$ is neither hyperelliptic or bielliptic, then $X$ has only finitely many points that are quadratic over $K$.
\end{thm}

In addition, we consider arithmetic questions regarding the fields of definition of quadratic points on $X_1(\calP)$. In particular, if $X_1(\calP)$ has a point defined over a quadratic number field $K$, what can be said about basic arithmetic invariants of $K$, such as discriminant and class number? For the curves $X_1^\Ell(N)$, arithmetic questions of this kind have been discussed by several authors: Momose \cite{momose:1984} shows that if $K$ is the field of definition of a quadratic point on $X_1^\Ell(13)$, then the prime 2 splits in $K$, and 3 is unramified in $K$; Bosman et al. \cite{bosman/etal:2014} show that $K$ must be a real quadratic field, an observation also made in \cite{doyle/faber/krumm:2014}. In the case of the modular curves $X_0^\Ell(N)$, Najman and Trbovi\'c \cite{najman/trbovic:2022} prove arithmetic results of this type for several values of $N$.

For the dynamical modular curves $X_1(\calP)$ we prove the following two theorems. Though our methods can be applied to several portraits in the set $\Gamma$ (namely, those for which the corresponding modular curve is hyperelliptic), the portraits 8(4) and 10(3,1,1) are highlighted here due to their significance in the context of elliptic curves, explained below.

By a {\bf quadratic point} on an algebraic curve over a field $k$, we mean a point whose field of definition is a quadratic extension of $k$.

\begin{thm}\label{main_class_group_thm} Let $\calP$ denote the portrait $8(4)$.
\begin{enumerate}
\item For every prime $p$ and every residue class $\frakc\in\bbZ/p\bbZ$, there exist infinitely many squarefree integers $d\in\frakc$ such that $X_1(\calP)$ has a quadratic point defined over $\bbQ(\sqrt d)$.
    \item There exist infinitely many imaginary quadratic fields $K$ with class number divisible by $10$ such that $X_1(\calP)$ has a quadratic point defined over $K$.
\end{enumerate}
\end{thm}

As noted in \cite{doyle/faber/krumm:2014}, the above curve $X_1(\calP)$ is isomorphic to $X_1^\Ell(16)$. Thus, Theorem \ref{main_class_group_thm} provides new information about the collection of quadratic fields $K$ such that there exists an elliptic curve $E/K$ with a $K$-rational torsion point of order 16.

Similarly, taking $\calP=10(3,1,1)$, the curve $X_1(\calP)$ is known to be isomorphic to $X_1^\Ell(18)$. The next theorem strengthens earlier results by Kenku--Momose \cite{kenku/momose:1988} regarding the splitting of rational primes in the fields of definition of quadratic points on this curve.

\begin{thm}\label{main_splitting_thm}
Let $\calP$ denote the portrait $10(3,1,1)$ and let $K$ be the field of definition of a quadratic point on $X_1(\calP)$.
\begin{enumerate}
\item The prime $2$ splits in $K$, and $3$ is not inert in $K$.
\item There exists an infinite and computable set of primes, denoted $\pi$, that is independent of $K$, and such that every prime in $\pi$ is unramified in $K$.
\end{enumerate}
\end{thm}

\subsection{Points of higher degree} Though our primary focus here is on quadratic fields, we make one observation concerning arbitrary number fields. The next result is a straightforward consequence of a theorem of Frey \cite{frey:1994} together with the main theorem of \cite{doyle/poonen:2020}. 

\begin{thm}\label{thm:degree_n_points}
Fix a positive integer $n$. For any portrait $\calP$, let $\gamma(\calP)$ denote the set of algebraic numbers $c\in\bbQ^{(n)}$ such that $\calP\cong G(f_c,K)$ for some number field $K$ satisfying $c\in K\subset\bbQ^{(n)}$.   There are only finitely many portraits $\calP$ such that $\gamma(\calP)$ is infinite.
\end{thm}

Note that Theorem~\ref{thm:main} is a more refined version of Theorem~\ref{thm:degree_n_points} in the case $n = 2$.

\subsection{Outline of the paper} In Section \ref{modular_curves_section} we define the notion of a \emph{generic quadratic portrait} and discuss basic facts concerning dynamical modular curves associated to such portraits, followed by general properties of algebraic curves in Section~\ref{sec:curve_properties}.

In Section \ref{infty_pts_section}, we apply geometric arguments to determine all generic quadratic portraits $\calP$ for which the curve $X_1(\calP)$ has infinitely many quadratic points. This proves, in particular, the implication (i) $\Rightarrow$ (ii) in Theorem~\ref{thm:main}; see Theorem~\ref{thm:inf_quad_pts} and the immediately preceding discussion.

Section~\ref{sec:classification} addresses the issue that $K$-rational points on $X_1(\calP)$ correspond to instances where $G(f_c,K)$ simply {\it contains} the portrait $\calP$; that is, we need not have an isomorphism $G(f_c,K) \cong \calP$, and in fact, in many cases we do not. The section culminates with the proofs of Theorems~\ref{main_thm_rational} and \ref{main_thm_quadratic} in \textsection\ref{sec:main_proofs}.

Finally, Section \ref{FOD_section} is devoted to arithmetic questions concerning the fields of definition of quadratic points on the curves $X_1(\calP)$, and in particular to proving Theorems \ref{main_class_group_thm} and \ref{main_splitting_thm}.

\subsection*{Acknowledgements} 
We thank Joe Silverman for a suggestion that led to the more refined statements in Theorems~\ref{main_thm_rational} and~\ref{main_thm_quadratic}, Joe Silverman and Borys Kadets for helpful comments on an earlier draft, and Mohammad Sadek for posing a question that led us to strengthen Proposition~\ref{prop:X0andX1}. The first author was partially supported by NSF grant DMS-2112697.

\section{Dynamical modular curves}\label{modular_curves_section}

\subsection{Dynatomic polynomials}
If $f$ is a polynomial
with coefficients in a field $K$ and $\alpha \in \Kbar$ is a point of exact period $n$ for $f$, then $\alpha$ is a root of the polynomial $f^n(z) - z$.  However, the roots of $f^n(z) - z$ may have period strictly dividing $n$, and indeed there is a factorization
	\[
		f^n(z) - z = \prod_{d\mid n} \Phi_{d,f}(z),
	\]
where (generically) the roots of $\Phi_{d,f}$ have exact period $d$ for $f$. M\"obius inversion yields
	\[
		\Phi_{n,f}(z) = \prod_{d \mid n} (f^d(z) - z)^{\mu(n/d)},
	\]
where $\mu$ denotes the M\"obius function. We call $\Phi_{n,f}$ the {\bf $n$th dynatomic polynomial} of $f$. 

More generally, for $m,n \ge 1$ we define
	\[
		\Phi_{m,n,f}(z) := \frac{\Phi_{n,f}(f^m(z))}{\Phi_{n,f}(f^{m-1}(z))}.
	\]
Then $\Phi_{m,n,f}$ is a polynomial whose roots are (again, generically) points of preperiod $m$ and eventual period $n$ for $f$. (That $\Phi_{n,f}$ and $\Phi_{m,n,f}$ are indeed polynomials is proven in \cite{silverman:2007, hutz:2015}.)

Since we are specifically interested in the family $f_c(z) = z^2 + c$, we write
	\[
		\Phi_n(c,z) := \Phi_{n,f_c}(z) \quad\text{and}\quad \Phi_{m,n}(c,z) := \Phi_{m,n,f_c}(z).
	\]
Then $\Phi_n$ (resp., $\Phi_{m,n}$) is a polynomial in $\bbZ[c,z]$, and the vanishing locus defines an affine curve $Y_1(n)$ (resp., $Y_1(m,n)$), which we refer to as a {\bf dynatomic curve}. Thus, for example, if $\alpha$ has period $n$ for $f_c$, then $(c,\alpha)$ is a point on the dynatomic curve $Y_1(n)$. We denote by $X_1(\cdot)$ the normalization of the projective closure of the affine curve $Y_1(\cdot)$, and we also refer to $X_1(\cdot)$ as a dynatomic curve.

\subsection{Dynamical modular curves associated to portraits}
The dynamical properties of quadratic polynomial maps impose certain restrictions on those portraits that may be realized as $G(f,K)$ for some number field $K$ and a quadratic polynomial $f \in K[z]$.
First, no point may have more than two preimages under $f$. Also, for each positive $n \in \bbZ$, the $n$th dynatomic polynomial for a quadratic polynomial $f$ has degree
    \begin{equation}\label{eq:D(n)}
        D(n) := \deg \Phi_{n,f}(z) = \sum_{d\mid n} \mu(n/d)2^d.
    \end{equation}
Thus, a quadratic polynomial has at most $D(n)$ points of period $n$, partitioned into at most $R(n) := D(n)/n$ cycles of length $n$. With these restrictions in mind, we make the following definition:
\pagebreak
\begin{defn}\label{defn:quad_portrait}
A {\bf quadratic portrait} is a portrait $\calP$ satisfying the following properties:
    \begin{enumerate}
        \item Every vertex of $\calP$ has in-degree at most $2$.
        \item For each $n\ge 1$, the number of $n$-cycles in $\calP$ is at most
            \[
                R(n) := \frac1n \sum_{d \mid n} \mu(n/d)2^d.
            \]
    \end{enumerate}
\end{defn}

For any number field $K$ and quadratic polynomial $f \in K[z]$, the portrait $G(f,K)$ is quadratic. However, for {\it most} quadratic polynomials (in a sense that can be made precise), we can say more about the structure of the set of $K$-rational preperiodic points. For the model $f_c(z) = z^2 + c$, if $\alpha$ is a preperiodic point for $f_c$, then $-\alpha$ is also preperiodic, since both are preimages of $f(\alpha)$. Thus, a preperiodic point typically has either no $K$-rational preimages or exactly two $K$-rational preimages. The exception to this rule occurs when $\alpha = 0$ is a preperiodic point, in which case exactly one preperiodic point (namely, $c = f_c(0)$) has a single $K$-rational preimage.

Along the same lines, if a polynomial $f_c$ has a $K$-rational fixed point $\beta$, then $\beta$ is a root of the quadratic polynomial $f_c(z) - z = z^2 - z + c$; thus, unless we have $\disc(f_c(z) - z) = 1 - 4c = 0$ (i.e., $c = 1/4$), there is a second fixed point $\beta'$, necessarily defined over $K$.

With these observations in mind, we make the following definition.

\begin{defn}\label{defn:generic_quadratic}
A {\bf generic quadratic portrait} is a quadratic portrait $\calP$ with the following additional properties:
    \begin{enumerate}
        \item The in-degree of any vertex of $\calP$ is equal to $0$ or $2$.
        \item If $\calP$ has a fixed point, then $\calP$ has exactly two fixed points.
    \end{enumerate}
\end{defn}

\begin{rem}
We will sometimes refer to the results of \cite{doyle:2019}, in which the term ``strongly admissible" is used instead of ``generic quadratic." 
\end{rem}

Given a quadratic portrait $\calP$, there is a {\bf dynamical modular curve} $Y_1(\calP)$, defined over $\bbQ$, whose $K$-points---for any extension $K/\bbQ$---correspond to tuples $(c,z_1,\ldots,z_n)$ such that $z_1,\ldots,z_n$ are preperiodic points forming a subportrait of $G(f_c,K)$ isomorphic to $\calP$. If $Q$ is the point on $Y_1(\calP)$ corresponding to such a tuple, then the field of definition of $Q$ is $\bbQ(c,z_1,\ldots, z_n)$. We denote by $X_1(\calP)$ the smooth projective curve birational to $Y_1(\calP)$.

A formal treatment of dynamical modular curves appears in \cite{doyle:2019}, where the curves are defined only for generic quadratic portraits. We lose no generality in making such a restriction: Given {\it any} quadratic portrait $\calP$, 
there is a unique portrait $\calP'$ that is minimal among generic quadratic portraits containing $\calP$ as a subportrait.
In the language of \cite{doyle:2019}, $\calP'$ is the generic quadratic portrait {\it generated by} the vertices of $\calP$. It follows from the results of \cite[\textsection 2]{doyle:2019} that $X_1(\calP) \cong X_1(\calP')$, so we 
may as well assume that $\calP$ is generic quadratic.

Rather than formally defining $X_1(\calP)$ (we refer the interested reader to \cite{doyle:2019} or, for a different approach in a more general setting, \cite{doyle/silverman:2020}), we give an example.

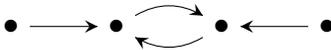
\begin{figure}
\centering
    \begin{tikzpicture}[scale=1.4]
\tikzset{vertex/.style = {}}
\tikzset{every loop/.style={min distance=10mm,in=45,out=-45,->}}
\tikzset{edge/.style={decoration={markings,mark=at position 1 with %
    {\arrow[scale=1.5,>=stealth]{>}}},postaction={decorate}}}
% vertices
\node[vertex] (3) at  (0, 0) {$\bullet$};
\node[vertex] (4) at  (1, 0) {$\bullet$};
\node[vertex] (5) at (2, 0) {$\bullet$};
\node[vertex] (6) at (3, 0) {$\bullet$};

% edges
\draw[edge] (3) to (4);
\draw[edge] (4) to[bend left=30] (5);
\draw[edge] (5) to[bend left=30] (4);
\draw[edge] (6) to (5);
\end{tikzpicture}
	\caption{A generic quadratic portrait}
	\label{fig:quadratic}
\end{figure}

\begin{ex}\label{ex:period2eqns}
Consider the generic quadratic portrait $\calP$ appearing in Figure~\ref{fig:quadratic}. One could construct a curve birational to $X_1(\calP)$ simply by giving one equation for each relation coming from an edge in $\calP$: if we label the vertices
$1,2,3,4$ from left to right, we have
	\begin{equation}\label{eq:curve_ex}
		\begin{split}
		&z_1^2 + c = z_2, \qquad z_2^2 + c = z_3,\\
		&z_3^2 + c = z_2, \qquad z_4^2 + c = z_3.
		\end{split}
	\end{equation}
Note that we must also impose certain Zariski open conditions of the form $z_i \ne z_j$ to remove extraneous components; 
for example, there is a full component of the curve defined by \eqref{eq:curve_ex} on which $z_1$, $z_2$, $z_3$, and $z_4$ are all equal. It is this approach that is taken in \cite{doyle/silverman:2020}.

An alternative approach, which is particular to quadratic polynomials, is to note that any generic quadratic portrait containing a point of 
period $2$ must necessarily contain $\calP$. Thus, another (affine) model for $X_1(\calP)$ is the plane curve defined by the vanishing of
	\[
		\Phi_2(c,z) = \frac{(z^2 + c)^2 + c - z}{z^2 + c - z} = z^2 + z + c + 1.
	\]
In other words, $X_1(\calP)$ is isomorphic to the dynatomic curve $X_1(2)$.
This second model has the advantage of being defined in a lower-dimensional affine space
($\bbA^2$, rather than $\bbA^5$), and it is this second approach which is described in detail in~\cite{doyle:2019}.

We conclude this example by pointing out that $X_1(\calP) \cong X_1(2)$ also has ``degenerate" points where two or more of the vertices of the portrait $\calP$ collapse. For example, the equation $\Phi_2(c,z) = 0$ has the solution $(c,z) = \left(-\frac{3}{4}, -\frac{1}{2}\right)$ despite the fact that $-\frac{1}{2}$ is a fixed point for $f_{-3/4}$. However, for a given portrait $\calP$, there are only finitely many such degenerate points on $X_1(\calP)$.
\end{ex}

Before summarizing the required properties of dynamical modular curves, we recall the following terminology:
\begin{defn}
Let $X$ be a smooth, irreducible projective curve defined over a field $k$. The {\bf $\boldsymbol{k}$-gonality} of $X$, denoted $\gon_k(X)$, is the minimal degree of a nonconstant morphism $X \to \bbP^1$ defined over $k$.
\end{defn}

\begin{prop}\label{prop:dmc_properties}
Let $\calP$ be a generic quadratic portrait, and let $k$ be any field of characteristic $0$.
    \begin{enumerate}
        \item The curve $X_1(\calP)$ is irreducible over $k$.
        \item If $\calP'$ is a generic quadratic portrait properly contained in $\calP$, then there is a finite morphism $\pi_{\calP,\calP'} : X_1(\calP) \to X_1(\calP')$ of degree at least $2$ defined over $k$.
        \item Given any ordering $\calP_1,\calP_2,\ldots$ of all generic quadratic portraits, the $k$-gonalities of the curves $X_1(\calP_i)$ tend to $\infty$.
    \end{enumerate}
\end{prop}

\begin{proof}
Parts (a) and (b) are proven in \cite[Thm. 1.7]{doyle:2019} and \cite[Prop. 3.3]{doyle:2019}, respectively. Note that the morphism $\pi_{\calP,\calP'}$ is obtained simply by forgetting the preperiodic points corresponding to vertices of $\calP \smallsetminus \calP'$, hence is defined over the base field $k$.

Statement (c) is a slight generalization of, but follows directly from, \cite[Thm. 1.1(b)]{doyle/poonen:2020}, which says that as $m + n \to \infty$, the gonalities of the curves $X_1(m,n)$ tend to $\infty$. Given a bound $B$, there are only finitely many quadratic portraits $\calP$ such that {\it every} vertex $v$ of $\calP$ has preperiod $m$ and eventual period $n$ satisfying $m + n \le B$. In other words, if for every generic quadratic portrait $\calP$ we choose a vertex $v_\calP$ with preperiod $m_\calP$ and eventual period $n_\calP$ maximizing the sum $m_\calP + n_\calP$, we must have $m_\calP + n_\calP \to \infty$ as $\calP$ ranges over all generic quadratic portraits in any order. Since there is a nonconstant morphism from $X_1(\calP)$ to $X_1(m_\calP, n_\calP)$ (e.g., by part (b)), we have $\gon_k(X_1(\calP)) \ge \gon_k(X_1(m_\calP, n_\calP))$, and the latter expression tends to $\infty$.
\end{proof}

\begin{proof}[Proof of Theorem~\ref{thm:degree_n_points}]
Fix $n \ge 1$ and a portrait $\calP$, and suppose there are infinitely many $c \in \bbQ^{(n)}$ such that $G(f_c,K) \cong \calP$ for some degree-$n$ number field $K$ containing $c$. Then the dynamical modular curve $X_1(\calP)$ has infinitely many points of degree at most $n$. It follows from \cite[Prop. 2]{frey:1994} (cf. \cite[Thm. 5]{clark:2009}) that $X_1(\calP)$ must have gonality at most $2n$, hence there are only finitely many such portraits $\calP$ by part (c) of Proposition~\ref{prop:dmc_properties}.
\end{proof}

\section{Some useful properties of algebraic curves}\label{sec:curve_properties}
In this section, we collect a few facts about algebraic curves that will be used throughout the rest of the paper.

First, we provide a statement that follows from Hilbert's irreducibility theorem; see \cite[\textsection 3.4]{serre:2008} and \cite[\textsection 9.2]{lang:1983} for details.

\begin{prop}
\label{prop:HIT}
Let $K$ be a number field, let $X$ be a curve defined over $K$, and let $\varphi : X \to \bbP^1$ be a dominant morphism of degree $d \ge 2$ defined over $K$. Then the set
\[
     \calT := \left\{P \in \bbP^1(K) : \left[K(Q) : K\right] < d \text{ for some } Q \in \varphi^{-1}(P)\right\}
\]
is a thin subset of $\bbP^1(K)$.
\end{prop}

\begin{rem}
Thin subsets $\calT \subset \bbP^1(K)$ have density $0$, in the sense that
\[
    \lim_{N\to\infty} \frac{\Big|\{P \in \calT : h(P) \le N\}\Big|}{\Big|\{P \in \bbP^1(K) : h(P) \le N\}\Big|}
            = 0,
\]
where $h$ is the na\"ive Weil height on $\bbP^1(\QQbar)$.
In particular, for any maximal ideal $\frakp \in \Spec \calO_K$ and any mod-$\frakp$ residue class $\frakc$ in $\bbP^1(K)$, the set $\frakc \setminus \calT$ is infinite.
\end{rem}

Given an elliptic curve $E$ with Weierstrass equation $y^2 = f(x)$, where $f \in K[x]$ is squarefree of degree $3$, it is easy to construct infinitely many quadratic points on $E$: For ``most" $x \in K$, the point $(x,y) = (x,\sqrt{f(x)})$ is quadratic over $K$. More precisely, since $f$ is not a square in $K[x]$, it follows from Hilbert irreducibility that $f(x_0)$ is a nonsquare in $K$ for all $x_0$ outside a thin subset of $K$. The following result, proven in \cite[Lem. 2.2]{doyle/faber/krumm:2014}, gives a useful characterization of quadratic points $(x,y)$ with $x \notin \bbQ$:

\begin{lem}\label{ell_quad} Let $E/K$ be an elliptic curve defined by an equation of the form \[y^2=ax^3+bx^2+cx+d,\] where $a,b,c,d\in K$ and $a\neq 0$. Suppose $(x,y)\in E(\bar K)$ is a quadratic point with $x\notin K$. Then there exist $(x_0,y_0)\in E(K)$ and $\V\in k$ such that $y=y_0+\V(x-x_0)$ and \[x^2 + \frac{ax_0 - \V^2 + b}{a}x + \frac{ax_0^2 + \V^2x_0 + bx_0 - 2y_0\V + c}{a}=0.\]
\end{lem}

By Theorem~\ref{thm:harris/silverman}, a curve with infinitely many quadratic points must admit a degree-$2$ morphism to either $\bbP^1$ or an elliptic curve, hence must have gonality at most $4$. Thus, to prove that a curve has finitely many quadratic points, it suffices
to show that the gonality of the curve is greater than $4$. 
The following inequality is a standard tool for finding lower bounds for gonalities.

\begin{prop}[{Castelnuovo--Severi inequality \cite[Thm. 3.11.3]{stichtenoth:2009}}]\label{prop:castelnuovo}
Let $Y$, $Y_1$, and $Y_2$ be curves of genera $g_Y$, $g_1$, and $g_2$, respectively. Suppose we have maps $\phi_1 : Y \to Y_1$ and $\phi_2 : Y \to Y_2$ of degrees $d_1$ and $d_2$, and suppose further that there is not an intermediate curve $Z$ and a map $\psi : Y \to Z$ of degree at least $2$ such that both $\phi_1$ and $\phi_2$ factor through $\psi$. Then
	\begin{equation}\label{eq:castelnuovo}
		g_Y \le d_1g_1 + d_2g_2 + (d_1 - 1)(d_2 - 1).
	\end{equation}
\end{prop}

\section{Dynamical modular curves with infinitely many quadratic points}\label{infty_pts_section}

The purpose of this section is to prove one direction of Theorem~\ref{thm:main}, 
namely that if there are infinitely many $c \in \bbQ^{(2)}$ such that $G(f_c,K) \cong \calP$ for some quadratic field $K$ containing~$c$,
then $\calP \in \Gamma$. Since any such realization of $\calP$ as $G(f_c,K)$ yields a quadratic point on the dynamical modular curve $Y_1(\calP)$, it suffices to prove the following:

\begin{thm}\label{thm:inf_quad_pts}
Let $\calP$ be a generic quadratic portrait. Then $X_1(\calP)$ has infinitely many quadratic points if and only if $\calP \in \Gamma$.
\end{thm}

\begin{rem}
If we just assume that $\calP$ is a quadratic portrait (i.e., not necessarily generic), then $X_1(\calP)$ has infinitely many quadratic points if and only if $\calP$ is a subportrait of some portrait in $\Gamma$. This follows from Theorem~\ref{thm:inf_quad_pts} as well as the fact that for any quadratic portrait $\calP$, if we let $\calP'$ be the minimal generic quadratic portrait containing $\calP$, then $X_1(\calP)$ and $X_1(\calP')$ are isomorphic over $\bbQ$.  (See the discussion preceding Example~\ref{ex:period2eqns}.)
\end{rem}

One direction of Theorem~\ref{thm:inf_quad_pts} is straightforward: For every portrait $\calP \in \Gamma$, the curve $X_1(\calP)$ is described in at least one of the articles \cite{walde/russo:1994,poonen:1998,morton:1998}. All the curves in those articles have genus at most $2$ and at least one rational point, hence have infinitely many quadratic points. Thus, we must show that if $\calP$ is generic quadratic but not in $\Gamma$, then $X_1(\calP)$ has only finitely many quadratic points.

To help organize the arguments in the rest of this section, we introduce some terminology:

\begin{defn}\label{defn:cycle_structure}
The {\bf cycle structure} of a portrait $\calP$ is the nonincreasing sequence of cycle lengths appearing in $\calP$. Note that the empty portrait has cycle structure $(\ )$.
\end{defn}

If $K$ is a quadratic field and $c \in K$, then the cycle structure of $G(f_c,K)$ may contain the integer $1$ at most twice and each of the integers $2$, $3$, and $4$ at most once; for periods~$1$ and~$2$ this follows from the fact that a quadratic polynomial can have at most two fixed points and at most one $2$-cycle, and for periods $3$ and $4$ this comes from \cite[Cor. 4.16]{doyle:2018quad}. More precisely, the results of \cite{doyle:2018quad} imply that the ``period at most $4$" portion of the cycle structure of $G(f_c,K)$ must be (4,1,1), (4,2), or one of the following:
	\begin{equation}\label{eq:cycle_structures}
	\text{
	(\ ),\ (1,1),\ (2),\ (3),\ (4),\ (2,1,1),\ (3,1,1),\ (3,2).
	}
	\end{equation}
Moreover, it follows from \cite[Cor. 3.48]{doyle/faber/krumm:2014} (resp., \cite[Thm. 4.21]{doyle:2018quad}) that no portrait with both a $4$-cycle and a $1$-cycle (resp., $4$-cycle and a $2$-cycle) may be realized infinitely often as $G(f_c,K)$ for $K$ a quadratic field and $c \in K$. For our purposes, therefore, we may exclude the cycle structures (4,1,1) and (4,2) from consideration.

By enumerating generic quadratic portraits with few vertices, one can verify that if $\calP$ is a generic quadratic portrait which is not in $\Gamma$, then $\calP$ has a cycle of length $n \ge 5$ or $\calP$ properly contains a portrait in $\Gamma_\rat$ or $\Gamma_\Quad$. We handle these two possibilities separately, showing in each case that the dynamical modular curve $X_1(\calP)$ has only finitely many quadratic points.

\subsection{Points of period $n \ge 5$}
If $\calP$ is a generic portrait with a cycle of length $n$, then there is a dominant morphism $X_1(\calP) \to X_1(n)$ defined over $\bbQ$. In particular, every quadratic point on $X_1(\calP)$ maps to a rational or quadratic point on $X_1(n)$, so we need only show that if $n \ge 5$, then $X_1(n)$ has only finitely many points defined over quadratic fields; we prove this in Proposition~\ref{prop:X0andX1}.

For $n \ge 1$, the cyclic group $C_n$ acts on $X_1(n)$ as follows: Given a point $(c,z) \in X_1(n)$, we also have $\sigma_n(c,z) := (c,f_c(z)) \in X_1(n)$, so $\sigma_n$ defines an order-$n$ automorphism of $X_1(n)$. We denote by $X_0(n)$ the quotient of $X_1(n)$ by this cyclic group action, which parametrizes maps $f_c$ together with a marked {\it cycle} of length $n$. Note that $c : X_1(n) \to \bbP^1$ factors through the quotient map $X_1(n) \to X_0(n)$; to avoid confusion, we denote by $c_0$ the corresponding map $X_0(n) \to \bbP^1$.

For a curve $X$, we will denote by $g_X$ its genus; for simplicity, for each $n \ge 1$ we will write $g_{0,n}$ for the genus of $X_0(n)$. Finally, recall that $D(n)$ denotes the degree (in $z$) of the polynomial $\Phi_n(c,z)$---equivalently, $D(n)$ is the degree of the morphism $c : X_1(n) \to \bbP^1$. As in Definition~\ref{defn:quad_portrait}, we let $R(n) := D(n)/n$. Note that $R(n)$ is the degree of ${c_0 : X_0(n) \to \bbP^1}$ and, since $D(n) \le 2^n$ for all $n \ge 1$, we have $R(n) \le 2^n/n$.

\begin{lem}\label{lem:genus_bound}
For all $n \ge 8$, we have $g_{0,n} > R(n) + 1$.
\end{lem}

\begin{proof}
In \cite[Thm. 13]{morton:1996}, Morton gives an explicit formula for $g_{0,n}$, and with this formula one can check that the conclusion holds whenever $8 \le n \le 16$. Thus, we assume $n \ge 17$.

In the proof of \cite[Thm. 13]{morton:1996}, Morton also provides the lower bound
	\[
		g_{0,n} \ge \frac{3}{2} + \left(\frac{1}{4} - \frac{1}{n}\right)2^n - (n + 1)2^{n/2 - 1}.
	\]
This implies that
    \[
        g_{0,n} - 1 > \frac{2^n}{n} \cdot \left(\frac{n}{4} - 1 - n(n+1)2^{1 - n/2}\right).
    \]
Thus, since $R(n) = D(n)/n \le 2^n/n$, it suffices to show that $\frac{n}{4} - 1 - n(n+1)2^{1 - n/2} \ge 1$ for all $n \ge 17$. The desired inequality is equivalent to
    \[
        (n + 1)2^{3 - n/2} \le 1 - \frac{8}{n},
    \]
and it is a calculus exercise to show that, for all $n \ge 17$,
    \[
        (n + 1)2^{3-n/2} < \frac{1}{2} < 1 - \frac{8}{n}.\qedhere
    \]
\end{proof}

\begin{prop}\label{prop:X0andX1}
Let $n \ge 1$.
    \begin{enumerate}
        \item $X_0(n)$ has infinitely many quadratic points if and only if $n \le 5$.
        \item $X_1(n)$ has infinitely many quadratic points if and only if $n \le 4$.
    \end{enumerate}
\end{prop}

\begin{rem}
The fact that $X_0(n)$ (hence also $X_1(n)$) has only finitely many quadratic points for sufficiently large $n$ follows from Theorem~\ref{thm:harris/silverman}, together with \cite[Thm. 1.1]{doyle/poonen:2020}, which states that the gonality of $X_0(n)$ tends to infinity with $n$.
For our purposes, however, we need the more precise statement of Proposition~\ref{prop:X0andX1}.
\end{rem}

\begin{proof}[Proof of Proposition~\ref{prop:X0andX1}]
We first prove (a). For $n \le 4$, the curve $X_0(n)$ is isomorphic over $\bbQ$ to $\bbP^1$, hence has infinitely many quadratic points---as does the curve $X_0(5)$, which has genus $2$. We now show that if $n \ge 6$, then $X_0(n)$ has finitely many quadratic points.

By Theorem~\ref{thm:harris/silverman}, it suffices to show that for each $n \ge 6$, $X_0(n)$ does not admit a degree-$2$ morphism to a curve of genus at most $1$.
This statement was proven by Stoll \cite{stoll:2008} for $n = 6$, so we suppose that $n \ge 7$, and we let $\varphi$ be a degree-$d$ morphism to a curve $C$ of genus $g_C \le 1$. We claim that $d > 2$.

For $n = 7$, one can find a model for $X_0(7)$ using the same approach as Stoll for $X_0(6)$; doing so, we find a model of the form $F(c,t) = 0$ for a polynomial $F \in \bbQ[c,t]$ of degree $18$ in $t$ and degree $9$ in $c$. In particular, the morphism $t : X_0(7) \to \bbP^1$ has degree $9$.

If there does not exist an intermediate curve $Z$ and a morphism $\psi : X_0(7) \to Z$ of degree at least $2$ such that both $\varphi$ and $t$ factor through $\psi$, then, since $X_0(7)$ has genus $16$, the Castelnuovo--Severi inequality tells us that $16 \le dg_C + 8(d - 1)$, hence $d \ge \frac{24}{g_C + 8} > 2$. On the other hand, if there is such a morphism $\psi : X_0(7) \to Y$, then $\deg \psi$ must be divisible by $3$, hence $d = \deg\varphi$ must be as well.

Now let $n \ge 8$. It follows from \cite[\textsection 3, Thm. 3]{bousch:1992} that the Galois group of (the Galois closure of) the cover $c_0$ is the full symmetric group $S_{R(n)}$; in particular, 
$c_0$ does not admit a proper subcover $X_0(n) \to Z \to \bbP^1$, so we may apply the Castelnuovo--Severi inequality to $c_0$ and $\varphi$ to get
    \[
        g_{0,n} \le dg_C + (R(n) - 1)(d - 1) \le d + (R(n) - 1)(d - 1).
    \]
Since $R(n) - 1 < g_{0,n} - 2$ by Lemma~\ref{lem:genus_bound}, we have
    \[
        g_{0,n} < d + (g_{0,n} - 2)(d - 1) = d(g_{0,n} - 1) - g_{0,n} + 1,
    \]
hence $d > 2$.

We now turn to part (b). For $n \le 4$, the curve $X_1(n)$ has genus at most $2$, hence has infinitely many quadratic points. It remains to show that if $n \ge 5$, then $X_1(n)$ has finitely many quadratic points. Note that this follows from (a) for all $n \ge 6$, so we need only show that $X_1(5)$ has finitely many quadratic points. As in part (a), it suffices to show that $X_1(5)$ is not a double cover of a curve of genus at most $1$. Thus, we let $\varphi : X_1(5) \to \bbP^1$ be a degree-$d$ morphism to a curve of genus $g_C \le 1$, and we show that $d > 2$.

A calculation in \textsc{Magma} \cite{magma} shows
that the morphism $X_1(5) \to \bbP^1$ given by the dynatomic polynomial ${\Phi_2(c,z) = z^2 + z + c + 1}$ has degree $7$. If $\phi$ factors through $\Phi_2$, then $d \ge 7$. If not, then the Castelnuovo--Severi inequality applies to the morphisms $\varphi$ and $\Phi_2$, hence
	\[
		g_{1,5} \le dg_C + 6(d - 1).
	\]
The curve $X_1(5)$ has genus $g_{1,5} = 14$, and we assumed $g_C \le 1$, so it follows that $d > 2$.
\end{proof}

\begin{rem}
The fact that $\Phi_2$ has low degree on $X_1(5)$ seems related to the fact that, as polynomials in $\QQbar[c,z]$, the dynatomic polynomials $\Phi_2$ and $\Phi_5$ have no common zeros $(c,z)$.
In particular, all zeros and poles of $\Phi_2$ on $X_1(5)$ lie above $\infty$, which restricts the possible number of such points.
\end{rem}

\subsection{Generic quadratic portraits properly containing the portraits in $\Gamma_\rat$ and $\Gamma_\Quad$}\label{sec:generic_properly_containing}

By enumerating portraits with few vertices, one finds that any generic quadratic portrait $\calP$ 
that has its cycle structure listed in \eqref{eq:cycle_structures}, but which is not contained in $\Gamma$, must properly contain a portrait from $\Gamma_\rat$ or $\Gamma_Quad$, and moreover, $\calP$ must have a subportrait isomorphic to one of the following portraits:
	\begin{equation}\label{eq:list}
	\text{10(1,1)a/b, 10(2), 10(3)a/b, 10(4), 12(2,1,1)a/b, or $G_n$ for some $1 \le n \le 10$}.
	\end{equation}
All portraits listed above appear in Appendix~\ref{graphs_appendix} except 10(4), which is the label we give to the subportrait of 12(4) shown in Figure~\ref{fig:10(4)}.

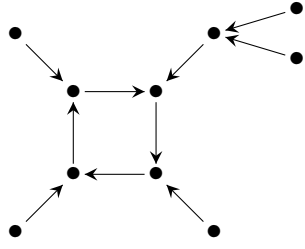
\begin{figure}\centering
    \begin{tikzpicture}[scale=1.1]
\tikzset{vertex/.style = {}}
\tikzset{every loop/.style={min distance=10mm,in=45,out=-45,->}}
\tikzset{edge/.style={decoration={markings,mark=at position 1 with %
    {\arrow[scale=1.5,>=stealth]{>}}},postaction={decorate}}}
    
% vertices
\node[inner sep=.4mm] (ul1) at  (-.7, 1.7) {$\bullet$};
\node[inner sep=.4mm] (ul0) at  (0, 1) {$\bullet$};
\node[inner sep=.4mm] (ur1) at  (1.7, 1.7) {$\bullet$};
\node[inner sep=.4mm] (ur0) at  (1, 1) {$\bullet$};
\node[inner sep=.4mm] (dl1) at (-.7, -.7) {$\bullet$};
\node[inner sep=.4mm] (dl0) at (0, 0) {$\bullet$};
\node[inner sep=.4mm] (dr1) at (1.7, -.7) {$\bullet$};
\node[inner sep=.4mm] (dr0) at (1, 0) {$\bullet$};
\node[inner sep=.4mm] (ur21) at (2.7, 1.4) {$\bullet$};
\node[inner sep=.4mm] (ur22) at (2.7, 2) {$\bullet$};

% This extra node is to center the main portion of the graph over the caption.
\node[vertex, white] (w) at  (-1.7, 1.7) {$\bullet$};

% edges
\draw[edge] (ul1) to (ul0);
\draw[edge] (ur1) to (ur0);
\draw[edge] (dl1) to (dl0);
\draw[edge] (dr1) to (dr0);
\draw[edge] (ur21) to (ur1);
\draw[edge] (ur22) to (ur1);
\draw[edge] (dr0) to (dl0);
\draw[edge] (ur0) to (dr0);
\draw[edge] (ul0) to (ur0);
\draw[edge] (dl0) to (ul0);
\end{tikzpicture}
	\caption{The portrait 10(4)}
	\label{fig:10(4)}
\end{figure}

\begin{prop}\label{prop:fin_quad_pts}
For each of the portraits $\calP$ appearing in \eqref{eq:list}, the curve $X_1(\calP)$ has only finitely many quadratic points.
\end{prop}

The cases $\calP = \rm 10(1,1)b$ and $\calP = \rm 10(2)$ form the majority of the proof of Proposition~\ref{prop:fin_quad_pts}. We include only the proof for $\rm10(1,1)b$, as the argument for $10(2)$ is very similar.

\begin{lem}\label{1011b_curve} Let $C\subset\Spec\bbQ[x,z]$ be the curve of genus $5$ defined by the equation
\[\left(z^2-2(x^2+1)\right)^2=2(x^2-1)^2(x^3+x^2-x+1).\]
Let $(c,p) \in \bbA^2(K)$ be such that $p$ has preperiod $4$ and eventual period $1$ for $f_c$. Then there exists a point $(x,z)\in C(K)$ such that $c=-2(x^2+1)/(x^2-1)^2$.
\end{lem}

\begin{proof}
Let $q=f_c(p)=p^2+c$, so that $q$ has preperiod $3$ (and still eventual period $1$). A calculation in \cite[p. 22]{poonen:1998} shows that there is an element $x\in K\smallsetminus\{\pm 1\}$ such that 
\[c=\frac{-2(x^2+1)}{(x^2-1)^2}\;\;\text{and}\;\; q^2=\frac{2(x^3+x^2-x+1)}{(x^2-1)^2}.\]
Hence we have \[2(x^3+x^2-x+1)=q^2(x^2-1)^2=(p^2+c)^2(x^2-1)^2=\left(\frac{p^2(x^2-1)^2-2(x^2+1)}{x^2-1}\right)^2.\] Letting $z=p(x^2-1)$ we obtain 
\[\left(z^2-2(x^2+1)\right)^2=(x^2-1)^2\cdot 2(x^3+x^2-x+1)\]
with $x,z\in K$. Thus $(x,z)\in C(K)$.
\end{proof}

\begin{prop}\label{1011b_finiteness}
For the portraits $\calP=\rm10(1,1)b$ and $\calP=10(2)$, the set of quadratic points on $X_1(\calP)$ is finite.
\end{prop}

\begin{proof}
As mentioned above, we only give a proof for $\calP=\rm10(1,1)b$. By Lemma \ref{1011b_curve}, it suffices to show that the curve $C$ has only finitely many quadratic points. The latter curve admits a dominant map to the elliptic curve with Weierstrass equation $w^2=2(x^3+x^2-x+1)$, which is the modular curve $X_1^\Ell(11)$. Explicitly, a natural map $\phi:C\to X_1^\Ell(11)$ is given by \[\phi(x,z)=\left(x,\frac{z^2-2(x^2+1)}{x^2-1}\right).\] 

The curve $X_1^\Ell(11)$ has exactly four affine rational points, namely, $(\pm 1,\pm 2)$. Suppose that $(x,z)$ is a quadratic point on $C$ with field of definition $K$. Then $x\notin\{\pm 1\}$, so the point $\phi(x,z)$ cannot be a rational point on $X_1(11)$. Thus $\phi(x,z)$ is a quadratic point, and $K=\bbQ(\phi(x,z))$. Letting
\begin{equation}\label{wz_eqn}
	w=\frac{z^2-2(x^2+1)}{x^2-1},
\end{equation}
we therefore have $w^2=2(x^3+x^2-x+1)$ and $K=\bbQ(x,w)$. We now consider two cases.

Suppose first that $x\in\bbQ$. Then $K=\bbQ(w)$, and by \eqref{wz_eqn} we have $z^2=2(x^2+1) + (x^2-1)w$. Applying the norm map $N_{K/\bbQ}$ to this equation we obtain
\[ y^2=4(x^2+1)^2-2(x^2-1)^2(x^3+x^2-x+1),\]
where $y=N_{K/\bbQ}(z)$. The above equation defines a hyperelliptic curve of genus 3, and therefore has only finitely many rational solutions. We conclude that $C$ has only finitely many quadratic points with rational $x$-coordinate.

Now suppose that $x\notin\bbQ$, so that $K=\bbQ(x)$. By Lemma \ref{ell_quad} applied to the equation $w^2=2(x^3+x^2-x+1)$, there is a rational number $\V$ and a point $(x_0,w_0)\in\{(\pm 1,\pm 2)\}$ such that $w=w_0+\V(x-x_0)$ and
\begin{equation}\label{10_11b_case2} 
	x^2 + \frac{2x_0-\V^2+2}{2}x + \frac{2x_0^2+\V^2x_0+2x_0-2w_0\V-2}{2}= 0.
\end{equation}

For each point $(x_0,w_0)\in\{(\pm 1,\pm 2)\}$ we consider the relation
\begin{equation}\label{10_11b_rel}
	z^2=2(x^2+1) + (x^2-1)(w_0+\V(x-x_0)).
\end{equation}
Using \eqref{10_11b_case2} we express the right-hand side of \eqref{10_11b_rel} as a linear combination of 1 and $x$. Applying the norm map $N_{K/\bbQ}$ and letting $u=2\cdot N_{K/\bbQ}(z)$, we obtain a relation of the form $u^2=g(\V)$, where $g$ is a polynomial of degree 7 with integral coefficients and nonzero discriminant. Each of the resulting four equations $u^2=g(\V)$ defines a hyperelliptic curve of genus 3, and therefore has only finitely many rational solutions. Since $\V$ has only finitely many possible values, \eqref{10_11b_case2} implies the same for $x$. Therefore $C$ has finitely quadratic points with quadratic $x$-coordinate.
\end{proof}

\begin{proof}[Proof of Proposition~\ref{prop:fin_quad_pts}]
The proposition has already been proven in \cite{doyle/faber/krumm:2014} and \cite{doyle:2018quad} for all of the portraits except $\rm10(1,1)b, 10(2)$, and $\rm10(3)a/b$. Moreover, Proposition~\ref{1011b_finiteness} shows that the statement is true for the portraits 10(1,1)b and 10(2), so all that remains is to show that $X_1(\calP)$ has only finitely many quadratic points when $\calP = \rm 10(3)a$ or $\calP = \rm 10(3)b$.

Each of the curves $X_1(\calP)$ with $\calP = \rm 10(3)a/b$ has genus $9$, and each admits a degree-$2$ map $\phi$ to the genus-$2$ curve $X_1(\calP')$, where $\calP' = \rm 8(3)$. Now suppose we have a degree-$d$ map $\psi : X_1(\calP) \to C$, where $C$ is a curve of genus $g_C \le 1$. Then, by Proposition~\ref{prop:castelnuovo}, either $\psi$ factors through $\phi$, in which case $\deg \psi > \deg \phi = 2$, or the Castelnuovo--Severi inequality \eqref{eq:castelnuovo} applies to $\phi$ and $\psi$, in which case we have
\[4 + dg_C + (d - 1) \ge 9,\;\text{ hence }\;d \ge \frac{6}{g_C + 1} \ge 3.\]
It follows that $X_1(\calP)$ is not hyperelliptic or bielliptic, hence $X_1(\calP)$ has only finitely many quadratic points by Theorem~\ref{thm:harris/silverman}.
\end{proof}

\subsection{Proof of Theorem~\ref{thm:inf_quad_pts}}

We now combine Propositions~\ref{prop:X0andX1} and \ref{prop:fin_quad_pts} to complete the proof of Theorem~\ref{thm:inf_quad_pts}, which in turn proves one direction of Theorem~\ref{thm:main}.

\begin{proof}[Proof of Theorem~\ref{thm:inf_quad_pts}]
As mentioned previously, the fact that $X_1(\calP)$ has infinitely many quadratic points for each $\calP \in \Gamma$ follows from the work of Walde--Russo \cite{walde/russo:1994} and Poonen \cite{poonen:1998}. Now suppose $\calP$ is 
a generic quadratic portrait such that $X_1(\calP)$ has infinitely many quadratic points. Proposition~\ref{prop:X0andX1} asserts that $\calP$ cannot have a cycle of length $n \ge 5$; combining this with the paragraph following Definition~\ref{defn:cycle_structure}, the cycle structure of $\calP$ must be one of those appearing in \eqref{eq:cycle_structures}. By simply enumerating all small generic quadratic portraits with the allowable cycle structures, one finds that if $\calP$ is not contained in $\Gamma$, then $\calP$ has a subportrait isomorphic to one of the portraits listed in \eqref{eq:list}, hence there is a dominant morphism $X_1(\calP) \to X_1(\calP')$ for some $\calP'$ in that list. Proposition~\ref{prop:fin_quad_pts} shows that each such $X_1(\calP')$ has only finitely many quadratic points, hence $X_1(\calP)$ does as well.
\end{proof}

\section{Preperiodic portraits realized infinitely often over quadratic fields}\label{sec:classification}

In this section, we show that if $\calP \in \Gamma$, then there are infinitely many $c \in \bbQ$ such that $G(f_c,K) \cong \calP$ for some quadratic field $K$. We also determine for which portraits $\calP$ the same is true for infinitely many $c \in \bbQ^{(2)} \smallsetminus \bbQ$.

It follows from Theorem~\ref{thm:inf_quad_pts} that $\Gamma$ is precisely the set of generic quadratic portraits that can be realized infinitely often as a {\it subportrait} of $G(f_c,K)$; the difficulty is in proving that we infinitely often have equality. This step requires two main tools: The first is Hilbert irreducibility, and the second is a dynamical result giving an upper bound for the lengths of periodic cycles of maps over number fields that depends only on the primes of bad reduction of those maps; see, for example, \cite{zieve:1996,silverman:2007,pezda:1994}.

\begin{prop}\label{prop:zieve}
Let $K$ be a number field, and let $\frakp \in \Spec \calO_K$ be a prime ideal of norm $q$. There exists a bound $B := B(q)$ such that if $c \in K$ and $v_\frakp(c) \ge 0$, then $f_c$ has no $K$-rational points of period larger than $B$.
\end{prop}

We will also repeatedly use the observation that given any portrait $\calP$ and any bound~$C$, there are only finitely many generic portraits $\calP'$ that are minimal (relative to inclusion) among those generic portraits containing $\calP$ {\it and} have no cycles of length larger than $C$. This is more or less due to the fact that there are only finitely many generic quadratic portraits with a given number of vertices.

\subsection{Rational \texorpdfstring{$c$}{c}-values}
For any $c \in \bbQ$, there are infinitely many quadratic fields $K$ for which $G(f_c,K) \cong G(f_c,\bbQ)$. This is a consequence of Northcott's theorem: The set of preperiodic points for $f_c$ has bounded height, hence there are only finitely many preperiodic points which are quadratic over $\bbQ$, and therefore only finitely many quadratic fields over which $f_c$ gains {\it new} preperiodic points.

In particular, if a portrait $\calP$ is realized as $G(f_c,\bbQ)$ for infinitely many $c \in \bbQ$, then $\calP$ must also be realized as $G(f_c,K)$ for infinitely many $c \in \bbQ$ and, for each such $c$, infinitely many quadratic fields $K$. The portraits realized infinitely often over $\bbQ$ are precisely the portraits in $\Gamma_0$; this is the main result of \cite{faber:2015}.

A more interesting problem, then, is to determine the set of portraits $\calP$ for which there are infinitely many $c \in \bbQ$ with $G(f_c,\bbQ) \subsetneq \calP$ but $G(f_c,K) \cong \calP$ for some quadratic field~$K$.

\begin{prop}\label{prop:rat_pts}
Let $\calP \in \Gamma$. Then there exist infinitely many $c \in \bbQ$ such that $G(f_c,K) \cong \calP$ for some quadratic field $K$.
Moreover, if $\calP \in \Gamma \smallsetminus \{\emptyset, {\rm 6(3)}\}$, the infinitely many $c \in \bbQ$ may be chosen so that
    \[
        G(f_c,\bbQ) \subsetneq G(f_c,K) \cong \calP.
    \]
\end{prop}

\begin{rem}
The portraits $\emptyset$ and $\rm 6(3)$ are genuine exceptions to the second statement, as asserted in Theorem~\ref{main_thm_rational} and proven in \textsection\ref{sec:main_proofs}.
\end{rem}

\begin{proof}[Proof of Proposition~\ref{prop:rat_pts}]
It follows from the discussion preceding the statement of Proposition~\ref{prop:rat_pts} that for both $\calP = \emptyset$ and $\calP = {\rm 6(3)}$, which are elements of $\Gamma_0$, there are infinitely many $c \in \bbQ$ such that $G(f_c,K) \cong \calP$ for some quadratic field $K$.
We henceforth assume $\calP \in \Gamma \smallsetminus \{\emptyset, {\rm 6(3)}\}$ and prove the stronger statement that there are infinitely many $c \in \bbQ$ such that $G(f_c,\bbQ) \subsetneq G(f_c,K) \cong \calP$ for some quadratic field $K$.

There is a model for $X_1(\calP)$ of the form $y^2 = F(x)$, with $F(x) \in \bbQ[x]$ nonconstant and squarefree, such that the morphism $c : X_1(\calP) \to \bbP^1$ factors through $x : X_1(\calP) \to \bbP^1$. If $X_1(\calP)$ has genus $0$, this is because there is a proper (generic quadratic) subportrait $\calP' \subsetneq \calP$ for which the natural morphism
    \[
        \pi_{\calP, \calP'} : X_1(\calP) \longto X_1(\calP')
    \]
described in Proposition~\ref{prop:dmc_properties}(b) has degree exactly $2$; see Table~\ref{tab:deg2_maps} for the list of such pairs $(\calP,\calP')$.
For the curves of genus $1$ or $2$, explicit models 
are given in Appendix~\ref{app:curve_models}.

\begin{table}
    \caption{For each pair $(\calP,\calP')$, there is a degree-$2$ morphism ${X_1(\calP) \to X_1(\calP')}$ defined over $\bbQ$.}
    \label{tab:deg2_maps}
    \centering
    \begin{tabular}{|c||c|c|c|c|c|}
    \hline
    $\calP$ & 4(1,1) & 4(2) & 6(1,1) & 6(2) & 8(2,1,1) \\
    \hline
    $\calP'$ & $\emptyset$ & $\emptyset$ & 4(1,1) & 4(2) & 4(1,1) {\bf or} 4(2) \\
    \hline
    \end{tabular}
\end{table}

Now fix a portrait $\calP \in \Gamma \smallsetminus \{\emptyset, {\rm 6(3)}\}$, and let $y^2 = F(x)$ be the model for $X_1(\calP)$ described in the previous paragraph. Choose any value of $x_0 \in \bbQ$, and choose a prime $p \in \Spec \bbZ$ of good reduction for $c : X_1(\calP) \to \bbP^1$ such that $v_p(c(x_0)) \ge 0$. Let $B = B(p^2)$ be the bound from Proposition~\ref{prop:zieve}, let $\calP_1,\ldots,\calP_n$ be the generic quadratic portraits that properly contain $\calP$ and that have no cycles of length larger than $B$, and for each $i = 1,\ldots,n$ let $\pi_i := \pi_{\calP_i,\calP} : X_1(\calP_i) \to X_1(\calP)$ be the natural projection morphism from Proposition~\ref{prop:dmc_properties}(b).

By Hilbert's Irreducibility Theorem, the sets
    \[
        \{x \in \bbQ : \sqrt{F(x)} \in \bbQ\}
    \]
and, for each $i = 1,\ldots,n$,
    \[
        \left\{x \in \bbQ : \left[\bbQ\left(\pi_i^{-1}\left(x,\sqrt{F(x)}\right)\right) : \bbQ\right] \le 2\right\},
    \]
are thin subsets of $\bbP^1(\bbQ)$.
Since the residue class
\[[x_0]_p := \{x \in \bbP^1(\bbQ) : x \equiv x_0 \Mod p\}\] is not thin, there are infinitely many $x \in [x_0]_p$ such that $K := \bbQ\left(\sqrt{F(x)}\right)$ is a quadratic field and ${\left(x,\sqrt{F(x)}\right) \in X_1(\calP)(K)}$ does not lift to a $K$-rational point on $X_1(\calP_i)$ for any $i = 1,\ldots,n$.

Now let $c = c(x)$ for any of the infinitely many $x$ from the previous paragraph. Excluding at most finitely many $x \in [x_0]_p$, we may assume that $G(f_c,K)$ is a generic quadratic portrait. 
Since $c$ lifts to a quadratic point $Q$ on $X_1(\calP)$, we have
    \[
        G(f_c,\bbQ) \subsetneq \calP \subseteq G(f_c,K).
    \]
On the other hand, since $c$ does {\it not} lift to a quadratic point on $X_1(\calP_i)$ for any $i = 1,\ldots,n$, the portrait $G(f_c,K)$ is either isomorphic to $\calP$ or contains a cycle of length greater than $B$.
Finally, since $v_p(c) \ge 0$, $G(f_c,K)$ has no cycles of length greater than $B$, and thus we have
    \[
        G(f_c,\bbQ) \subsetneq G(f_c,K) \cong \calP.\qedhere
    \]
\end{proof}

\subsection{Quadratic \texorpdfstring{$c$}{c}-values}

The purpose of this section is to determine which portraits $\calP$ may be realized infinitely often as $G(f_c, \bbQ(c))$ for some quadratic algebraic number $c$. We begin with the simplest case, namely, where $X_1(\calP)$ has genus $0$.

\begin{prop}\label{prop:exact_genus0}
Suppose $\calP \in \Gamma_0$, and let $K/\bbQ$ be a quadratic field. Then there are infinitely many $c \in K \smallsetminus \bbQ$ such that $G(f_c,K) \cong \calP$.
\end{prop}

\begin{proof}
Let $X := X_1(\calP) \cong_\bbQ \bbP^1$. Choose an {\it inert} prime $p \in \Spec \bbZ$ such that $c : X \to \bbP^1$ has good reduction at $p$ and such that $p > D := \deg(c : X \to \bbP^1)$. Let $\frakp \in \Spec \calO_K$ be the unique prime lying above $p$. The good reduction condition implies that the mod-$\frakp$ reduction $\widetilde{c}$ of the map $c : X \to \bbP^1$ has degree $D$, and the condition that $D < p$ implies that $\widetilde{c}$ cannot map all of $\bbP^1(\bbF_{p^2})$ to $\bbP^1(\bbF_p)$. In other words, we may choose a point $P_0 \in X(K)$ such that $\widetilde{c(P_0)} = \widetilde{c}(\widetilde{P_0}) \in \bbP^1(\bbF_{p^2}) \smallsetminus \bbP^1(\bbF_p)$. Note that since $\infty \in \bbP^1(\bbF_p)$ and $\widetilde{c(P_0)} \notin \bbP^1(\bbF_p)$, we must have $v_\frakp(c(P_0)) \ge 0$.
Then, for any $P$ in the residue class $[P_0]_\frakp$, $c(P)$ must be in $K \smallsetminus \bbQ$, and $v_p(c(P)) \ge 0$. In particular, for all $P \in [P_0]_\frakp$, $f_{c(P)}$ has no $K$-rational points of period greater than $B = B(p^2)$, the bound from Proposition~\ref{prop:zieve}.

Let $\calP_1,\ldots,\calP_n$ be the complete list of generic quadratic portraits that minimally contain $\calP$ and which have no cycles of length larger than $B$, and for each $i = 1,\ldots,n$ let
    \[
        \pi_i := \pi_{\calP_i,\calP} : X_1(\calP_i) \longto \bbP^1
    \]
be the morphism from Proposition~\ref{prop:dmc_properties}(b).
If $P \in [P_0]_\calP$ and $G(f_{c(P)},K)$ properly contains $\calP$, then $G(f_{c(P)}, K)$ must contain one of the portraits $\calP_i$. By Hilbert irreducibility, the set
	\[
		[P_0]_\frakp \smallsetminus \bigcup_{i=1}^n \pi_i\big(X_1(\calP_i)(K)\big)
	\]
is infinite. Thus, there are infinitely many $c \in K \smallsetminus \bbQ$ such that $G(f_c,K) \cong \calP$.
\end{proof}

We now consider the portraits $\calP \in \Gamma \smallsetminus \Gamma_0$; that is, the portraits for which $X_1(\calP)$ has genus $1$ or $2$. We begin with a useful consequence of Theorem~\ref{thm:inf_quad_pts}.

\begin{cor}\label{cor:g1_g2}
Let $\calP$ be a generic quadratic portrait such that $X_1(\calP)$ has positive genus. If $\calP'$ is any generic quadratic portrait properly containing $\calP$, then $X_1(\calP')$ has finitely many quadratic points.
\end{cor}

\begin{proof}
If $X_1(\calP')$ has infinitely many quadratic points, then so does $X_1(\calP)$, and therefore both $\calP$ and $\calP'$ are elements of $\Gamma$. By simply inspecting the portraits in $\Gamma$, the only way we can have $\calP,\calP' \in \Gamma$ and $\calP \subsetneq \calP'$ is if $\calP \in \Gamma_0$; that is, if $X_1(\calP)$ has genus $0$.
\end{proof}

Combining Corollary~\ref{cor:g1_g2} and Proposition~\ref{prop:zieve} gives us the following sufficient condition for $\calP$ to be realized infinitely often as $G(f_c,K)$ over quadratic fields $K$. For an algebraic curve $X$ defined over a field $k$, we denote by $X(k,2)$ the set of all points on $X$ of degree at most $2$ over $k$. For a prime $p \in \Spec \bbZ$ and an element $\alpha \in \QQbar$, the phrase ``$v_p(\alpha) \ge 0$" should be read to mean ``there exists some extension $\frakp \in \Spec \calO_{\bbQ(\alpha)}$ of $p$ such that $v_\frakp(\alpha) \ge 0$."

\begin{lem}\label{lem:sufficient}
Let $\calP$ be a generic quadratic portrait such that $X_1(\calP)$ has genus $1$ or $2$. Fix a prime $p \in \Spec \bbZ$, and consider the set
    \[
    \calS_{p,\calP} := \{\beta \in X_1(\calP)(\bbQ,2) : v_p(c(\beta)) \ge 0\}.
    \]
For all but finitely many $\beta \in \calS_{p,\calP}$, we have 
        $G(f_{c(\beta)}, \bbQ(\beta)) \cong \calP.$
    
\end{lem}

\begin{proof}
Suppose $\beta \in \calS_{p,\calP}$, set $c := c(\beta)$, and let $K := \bbQ(\beta)$. Removing at most finitely many points $\beta \in \calS_{p,\calP}$, we may assume that $G(f_c,K)$ is generic quadratic.
Since ${\beta \in X_1(\calP)(K)}$ and $G(f_c,K)$ is generic quadratic, $G(f_c,K)$ has a subportrait isomorphic to~$\calP$.

Let $\frakp \in \Spec \calO_K$ be a prime lying above $p$. Since $\frakp$ has norm at most $p^2$, the map $f_c$ has no $K$-rational points of period larger than $B := B(p^2)$. Thus, as explained following Proposition~\ref{prop:zieve}, if $G(f_c,K) \not \cong \calP$, then $G(f_c,K)$ must contain one of finitely many portraits $\calP_1,\ldots,\calP_n$ properly containing $\calP$. But each of the curves $X_1(\calP_i)$ has only finitely many quadratic points by Corollary~\ref{cor:g1_g2}; this completes the proof.
\end{proof}

\begin{prop}\label{prop:onlyQQ}
Let $\calP\in\Gamma_\rat$. If $K$ is a quadratic field and $c \in K$ satisfies $G(f_c,K) \cong \calP$, then $c \in \bbQ$.
\end{prop}
\begin{proof}
For $\calP \in \{\mathrm{8(4),\ 10(3,1,1),\ 10(3,2)}\}$, this follows immediately from Theorems 3.16, 3.25, and 3.28 of \cite{doyle/faber/krumm:2014}. For the remaining portraits $\calP$---namely, 8(1,1)a and 8(2)a---the proofs are similar, so we only provide the details for 8(1,1)a.

Let $\calP = \rm8(1,1)a$. As shown in Appendix~\ref{app:curve_models}, $X_1(\calP)$ is isomorphic to the elliptic curve $E$ with affine model $y^2 = x^3 - x^2 + x$, which is the curve labeled 24A4 in \cite{cremona:1997} and 24.a5 in \cite{lmfdb}. We claim that if $P = (x,y)$ is a quadratic point on $E$, then $c(P) = -\frac{(x^2 + 1)^2}{4x(x-1)^2} \in \bbQ$. This is certainly the case if $x \in \bbQ$, so assume that $x$ is quadratic.

By Lemma~\ref{ell_quad}, there exist $P_0 = (x_0,y_0) \in E(\bbQ) \smallsetminus \infty$ and  $\V\in\bbQ$ such that
    \begin{equation}\label{eq:g1case1_minpoly}
        x^2 + (x_0 - \V^2 - 1)x + (x_0^2 + \V^2x_0 - x_0 - 2y_0\V + 1) = 0.
    \end{equation}
The only affine rational points $(x_0,y_0)$ on $E$ are $(0,0)$ and $(1,\pm 1)$, and for each of these points we can use \eqref{eq:g1case1_minpoly} to rewrite $c(P)$ as a function of $\V$ and $x$ which has degree at most $1$ in $x$. For the points $(0,0)$, $(1,1)$, and $(1,-1)$, respectively, we get
    \begin{align*}
        c &= -\frac{\left(\V^2+1\right)^2}{4 (\V-1) (\V+1)},\\
        c &= -\frac{\V^2 \left(\V^2-2 \V+2\right)}{4 (\V-1)^2}, \text{ and}\\
        c &= -\frac{\V^2 \left(\V^2+2 \V+2\right)}{4 (\V+1)^2}.
    \end{align*}
In any case, since $\V \in \bbQ$, we must also have $c \in \bbQ$.
\end{proof}

\begin{prop}\label{prop:inf_quad}
For all $\calP\in\Gamma_0\cup\Gamma_\Quad$ there exist infinitely many $c \in \bbQ^{(2)}$ such that
\[G(f_c, \bbQ(c)) \cong \calP.\]
\end{prop}
\begin{proof}
For the portraits in $\Gamma_0$, this follows from Proposition~\ref{prop:exact_genus0}. The proofs for 8(1,1)b and 8(2)b (resp., 10(2,1,1)a and 10(2,1,1)b) are very similar, so we only provide the details for
    \[
        \text{8(1,1)b, 10(2,1,1)a, and 8(3).}
    \]

First, let $\calP$ be the portrait 8(1,1)b. Appendix~\ref{app:curve_models} shows that $X_1(\calP)$ is isomorphic to the curve $X$ with affine model $y^2 = 2(x^3 + x^2 - x + 1)$, with $c : X \to \bbP^1$ given by
    \[
        c = -\frac{2(x^2 + 1)}{(x+1)^2(x-1)^2}.
    \]
Set $P_0 = (x_0, y_0) := (1,2) \in X(\bbQ)$. For $\V \in \bbQ$, the line $y - 2 = \V(x - 1)$ intersects the curve $X$ at $P_0$ and two additional points $P_\V$ and $\Pbar_\V$. Since $\V$ and $P_0$ are rational, either $P_\V$ and $\Pbar_\V$ are rational as well, or $P_\V$ and $\Pbar_\V$ are quadratic conjugates. Since $X(\bbQ)$ is finite, we may, at the expense of excluding finitely many $\V$, assume that $P_\V$ and $\Pbar_\V$ are quadratic Galois conjugates. Let $K := \bbQ(P_\V)$, and let $\tau$ be the nontrivial element of $\Gal(K/\bbQ)$. With this setup, we have $P_\V + \Pbar_\V = -P_0 \ne \calO$, so $y(P_\V) \ne -y(\Pbar_\V) = -y(P_\V)^\tau$, and therefore $x(P_\V) \notin \bbQ$. By calculating the intersection of $y - 2 = \V(x - 1)$ with the affine curve $X$, we find that the minimal polynomial of $x(P_\V)$ must be
    \[
    x^2 - \frac{\V^2 - 4}{2}x + \frac{\V^2 - 4\V + 2}{2}.
    \]
Thus, we may rewrite $c(P_\V)$ as
    \[
    c(P_\V) = \frac{\V+2}{8 (\V-2)} x(P_\V) - \frac{\V^5-10 \V^3+8 \V^2+8 \V+32}{16 (\V-2)^2 \V}.
    \]
Finally, for any $\V \in \bbQ$ with $\V \equiv 1 \Mod 3$, we see that $x(P_\V)$ and $c(P_\V)$ are integral at $p = 3$, so $P_\V \in S_{3,\calP}$. By Lemma~\ref{lem:sufficient}, we conclude that there are infinitely many quadratic $c$ with $G(f_c,\bbQ(c)) \cong \calP$.

Next, we let $\calP$ be the portrait 10(2,1,1)a, and we take $X$ to be the curve defined by ${y^2 + xy + y = x^3 - x^2 - x}$ with map $c : X \to \bbP^1$ given by
    \[
        c = \frac{x-2}{4x(x-1)}y - \frac{x^4 - x^3 + 3x - 1}{4x^2(x-1)}.
    \]
By Hilbert irreducibility, there are infinitely many points on $X$ with $x \in \bbQ$, $x \equiv 2 \Mod 3$, and $y \notin \bbQ$. For all such points $P = (x,y)$, $c(P)$ must be quadratic, and we have $P \in \calS_{3,\calP}$. The desired conclusion again follows from Lemma~\ref{lem:sufficient}.

Finally, we let $\calP$ be the portrait 8(3). Let $X$ be the curve $y^2 = x^6 - 2x^4 + 2x^3 + 5x^2 + 2x + 1$ with $c : X \to \bbP^1$ given by
    \[
        c = -\frac{x^6 + 2x^5 + 4x^4 + 8x^3 + 9x^2 + 4x + 1}{4x^2(x + 1)^2}.
    \]
The curve $X$ has two rational points at infinity; we denote these by $\infty^+$ and $\infty^-$. These two points are transposed by the hyperelliptic involution $\iota: X \to X$ given by $\iota(x,y) = (x,-y)$.

Let $J$ be the Jacobian of the genus-$2$ curve $X$. For a thorough treatment of the arithmetic of genus-$2$ curves, we recommend \cite{cassels/flynn:1996}; here, we just summarize the necessary properties. Points on $J$ correspond to degree-$0$ divisor classes on $X$. Moreover, by the Riemann--Roch theorem, every nontrivial divisor class can be written \textit{uniquely} as
    \[
    \{P,Q\} := [P + Q - \infty^+ - \infty^-]
    \]
with $P,Q \in X$ and $Q \ne \iota(P)$, up to swapping $P$ and $Q$. (The trivial class $\calO$ is equal to $\{P,\iota(P)\}$ for all $P \in X$.) A point $\{P,Q\} \ne \calO$ is rational if and only if either $P$ and $Q$ are both rational themselves, or $P$ and $Q$ are Galois-conjugate quadratic points.

Fix the point
    \[
        P_0 := \left(-\frac{1}{4}\left(1 + \sqrt{-15}\right), -\frac{1}{16}\left(17 + 9\sqrt{-15}\right)\right) \in X(\bbQ,2),
    \]
for which we have $c(P_0) = \frac{1}{48}\left(7 + 8\sqrt{-15}\right) \notin \bbQ$. If we let
    \[
        \Pbar_0 := \left(-\frac{1}{4}\left(1 - \sqrt{-15}\right), -\frac{1}{16}\left(17 - 9\sqrt{-15}\right)\right)
    \]
be the Galois conjugate of $P_0$, then $\calD_0 := \{P_0, \Pbar_0\} \in J(\bbQ)$. Note that $\Pbar_0 \ne \iota(P_0)$, so $\calD_0 \ne \calO$. Poonen showed in \cite[Prop. 1]{poonen:1998} that $J(\bbQ) \cong \bbZ$, so $\calD_0$ has infinite order.

The curve $X$ (hence also its Jacobian $J$) has good reduction at the prime $p = 7$. A straightforward computation (e.g., in \textsc{Magma}) shows that the reduction $\tilde{\calD}_0 \in J(\bbF_7)$ has order $21$, so $\calD_n := (1 + 21n)\calD_0 \equiv \calD_0 \Mod 7$ for all $n \in \bbZ$. For each $n$ we write ${\calD_n = \{P_n, \Pbar_n\}}$ with $P_n \equiv P_0 \Mod 7$ and $\Pbar_n \equiv \Pbar_0 \Mod 7$. We claim that $P_n \in \calS_{7,\calP}$ for all $n \in \bbZ$, from which the result follows by Lemma~\ref{lem:sufficient}.

Since $-15 \equiv -1 \Mod 7$ is not a square in $\bbF_7$, $\tilde{P_0}$ is quadratic over $\bbF_7$, thus the same is true for $\tilde{P_n}$ for all $n \in \bbZ$. This implies that $P_n$ is quadratic over $\bbQ$ for all $n$.

The map $c : X \to \bbP^1$ has good reduction at $p = 7$, so $P_n \equiv P_0 \Mod 7$ implies that $c(P_n) \equiv c(P_0) \Mod 7$. Arguing as in the previous paragraph, we conclude that $c(P_n)$ is quadratic over $\bbQ$. Finally, we note that since $c(P_n) \equiv c(P_0) \not\equiv \infty \Mod 7$, we have $v_7(c(P_n)) \ge 0$, so $P_n \in \calS_{7,\calP}$.
\end{proof}

\subsection{Proofs of Theorems~\ref{main_thm_rational} and \ref{main_thm_quadratic}}\label{sec:main_proofs}

\begin{proof}[Proof of Theorem~\ref{main_thm_rational}]
That (ii) implies (i) is precisely the second statement in Proposition~\ref{prop:rat_pts}, so it remains only to show that (i) implies (ii).

Assume there are infinitely many $c \in \bbQ$ such that $G(f_c,\bbQ) \subsetneq G(f_c,K) \cong \calP$ for some quadratic field $K$.
Every such occurrence of $\calP$ as $G(f_c,K)$ yields a quadratic point on $Y_1(\calP)$, so there are infinitely many quadratic points on $X_1(\calP)$. Thus $\calP \in \Gamma$ by Theorem~\ref{thm:inf_quad_pts}.

All that remains for us to show is that $\calP$ cannot be isomorphic to $\emptyset$ or $\rm 6(3)$. Certainly one cannot have $G(f_c,\bbQ) \subsetneq G(f_c,K) \cong \emptyset$, so we need only show that if $c \in \bbQ$ and $K$ is a quadratic field with $G(f_c,K) \cong \rm 6(3)$, then in fact $G(f_c,\bbQ) \cong \rm 6(3)$.

Supposing that $G(f_c,K) \cong {\rm 6(3)}$, the map $f_c$ then has a period-$3$ point $\alpha \in K$, and the six $K$-rational preperiodic points prescribed by the portrait $\rm 6(3)$ are $\pm \alpha$, $\pm f_c(\alpha)$, and $\pm f_c^2(\alpha)$. It therefore suffices to show that $\alpha \in \bbQ$.

Let $\tau$ be the nontrivial element of the Galois group $\Gal(K/\bbQ)$. Since $f_c$ is defined over $\bbQ$, $\alpha^\tau$ is also a $K$-rational point of period $3$, hence lies in the cycle $\{\alpha, f_c(\alpha), f_c^2(\alpha)\}$ (because the portrait 6(3) has only one $3$-cycle\footnote{In fact, if $K$ is any quadratic field and $c \in K$, then $f_c$ has at most one $3$-cycle defined pointwise over $K$. This follows from \cite[Thm. 3]{morton:1992} when $c \in \bbQ$ and \cite[Thm. 4.5]{doyle:2018quad} in general.}). Write $\alpha^\tau = f_c^k(\alpha)$ for some $k \in \{0,1,2\}$. Then
    \[
        \alpha = (\alpha^\tau)^\tau = f_c^k(\alpha^\tau) = f_c^{2k}(\alpha).
    \]
Since $\alpha$ has exact period $3$, this implies that $k = 0$; that is, $\alpha^\tau = \alpha$. Thus $\alpha \in \bbQ$, and therefore $G(f_c,\bbQ) \cong {\rm 6(3)}$.
\end{proof}

\begin{proof}[Proof of Theorem~\ref{main_thm_quadratic}]
First suppose that (i) holds; i.e., there are infinitely many ${c \in \bbQ^{(2)}\smallsetminus\bbQ}$ such that $G(f_c,\bbQ(c)) \cong \calP$. The curve $X_1(\calP)$ then has infinitely many quadratic points, and therefore $\calP \in \Gamma$ by Theorem~\ref{thm:inf_quad_pts}. By Proposition~\ref{prop:onlyQQ}, we must have
    \[
        \calP \in \Gamma \smallsetminus \Gamma_\rat = \Gamma_0 \cup \Gamma_\Quad.
    \]
Thus, (i) implies (ii). The converse is precisely Proposition~\ref{prop:inf_quad}.
\end{proof}

\section{Fields of definition of quadratic points}\label{FOD_section}

In this section we address the question of whether the existence of a given preperiodic portrait over a given quadratic field has implications regarding standard arithmetic invariants of that field. In particular, we focus on the portraits that occur infinitely often over quadratic fields, namely those in the set $\Gamma$.

To be precise, we are interested in arithmetic invariants of the fields of definition of quadratic points on dynamical modular curves $X_1(\calP)$. To partially justify the transition from \textit{realizations of portraits} to simply \textit{points on dynamical modular curves}, we begin by proving the following result:

\begin{prop}
For a number field $K$ and a generic quadratic portrait $\calP$, the following are equivalent:
    \begin{enumerate}[\tfae]
        \item There exist infinitely many $c \in K$ such that $G(f_c,K) \cong \calP$.
        \item The curve $X_1(\calP)$ has infinitely many $K$-rational points.
    \end{enumerate}
\end{prop}

\begin{proof}
That (i) implies (ii) is immediate from the definition of $X_1(\calP)$, so suppose that $X_1(\calP)(K)$ is infinite. 
Since $X_1(\calP)$ is irreducible (see Proposition~\ref{prop:dmc_properties}(a)), this implies that $X_1(\calP)$ is isomorphic over $K$ either to $\bbP^1$ or to an elliptic curve.

In the former case, the result follows from essentially the same proof as for Proposition~\ref{prop:exact_genus0}. In fact, the appropriate modification of that proof shows that there are infinitely many $c \in K$ such that $\bbQ(c) = K$ and such that $G(f_c,K) \cong \calP$.

In the latter case, choose a non-torsion point $Q_0 \in X_1(\calP)(K)$, and choose a prime $\frakp \in \Spec \calO_K$ of good reduction for the morphism $c : X_1(\calP) \to \bbP^1$ such that $v_\frakp(Q_0) \ge 0$. Since $Q_0$ is non-torsion, there are infinitely many points $Q \in X_1(\calP)(K)$ such that $Q \equiv Q_0 \Mod \frakp$, and since the morphism $c$ has good reduction at $\frakp$, we have $c(Q) \equiv c(Q_0) \Mod \frakp$ for every such $Q$. In particular, we have infinitely many $Q \in X_1(\calP)(K)$ such that $v_\frakp(c(Q)) \ge 0$, so the desired result follows from Lemma~\ref{lem:sufficient}.
\end{proof}

We now move on to a result concerning the splitting of rational primes in the quadratic fields over which the portrait $10(3,1,1)$ is realized as a preperiodic portrait. This example is included in order to illustrate our methods, but similar reasoning can be applied to any portrait in $\Gamma$ for which the corresponding modular curve is hyperelliptic.

As noted in \cite{poonen:1998}, the dynamical modular curve  corresponding to the portrait $10(3,1,1)$ is isomorphic to $X_1^\Ell(18)$. If $K$ is the field of definition of a quadratic point on this curve, Kenku and Momose \cite[Prop. 2.4]{kenku/momose:1988} show that 
\begin{itemize}
 \item either 2 splits or 3 does not split in $K$;
 \item 3 is not inert in $K$; and
 \item 5 and 7 are unramified in $K$.
 \end{itemize}
In what follows, for every polynomial $f\in\bbZ[x]$, we denote by $\pi_f$ the set of all integer primes $p$ such that $f$ does not have a root modulo $p$. Extending the results of Kenku and Momose, we prove the following. (Note that this proves Theorem \ref{main_splitting_thm}.)

\begin{thm}\label{10311_splitting} Let $K$ be a quadratic field such that ${G(f_c,K)\cong 10(3,1,1)}$ for some $c\in K$. Then the prime $2$ splits in $K$, $3$ is not inert in $K$, and letting \[f(x)=x^6 + 2x^5 +5x^4 + 10x^3 + 10x^2 + 4x +1,\] every prime in the set $\pi_f$ (which includes $5$ and $7$) is unramified in $K$. Moreover, $\pi_f$ has Dirichlet density $13/18$.
\end{thm}

For every nonzero rational number $r$, we let $\sqf(r)$ denote the squarefree part of $r$, i.e., the unique squarefree integer $d$ such that $r/d$ is the square of a rational number.

\begin{lem}\label{unram} Let $f\in\bbZ[x]$ be a monic polynomial of even degree, and let $p$ be an odd prime. If $p\in\pi_f$, then $p$ is unramified in every quadratic field of the form $\bbQ(\sqrt{f(r)})$ with $r\in\bbQ$.
\end{lem}

\begin{proof} Given a quadratic field $K=\bbQ(\sqrt{f(r)})$, we must show that $p$ does not divide the discriminant of $K$. Let $D=\sqf(f(r))$, so that $K=\bbQ(\sqrt D)$. Since $p$ is odd, it suffices to show that $p$ does not divide $D$. Set $g(x,y)=y^{2k}f(x/y)\in\bbZ[x,y]$, where $\deg(f)=2k$, and write $r=n/d$ with $\gcd(n,d)=1$. Then $D=\sqf(g(n,d))$, so that
\[g(n,d)=Ds^2,\quad s\in\bbZ.\]

We now consider two cases. If $d\equiv 0\bmod p$, then the above equation can be reduced modulo $p$ to obtain $n^{2k}\equiv Ds^2\bmod p$. Since $p$ cannot divide $n$ (given that $n$ and $d$ are coprime), we conclude that $p$ does not divide $D$, as required.

Suppose now that $d\not\equiv 0\bmod p$. We can then consider the equation $d^{2k}f(n/d)=Ds^2$ as taking place in the ring $\bbZ_p$. If $p\mid D$, then reducing modulo $p$ we obtain $f(n/d)\equiv 0\bmod p$, contradicting the hypothesis that $f$ has no root modulo $p$. Therefore $p$ cannot divide $D$.
\end{proof}

\begin{proof}[Proof of Theorem \ref{10311_splitting}]
By \cite[Thm. 3.25]{doyle/faber/krumm:2014}, we have ${K=\bbQ(\sqrt{f(r)})}$ for some $r\in\bbQ\smallsetminus\{0,-1\}$. Writing $r=n/d$ in lowest terms, it follows that $K=\bbQ(\sqrt{g(n,d)})$, where \[g(n,d) :=d^6f(n/d)=n^6+2n^5d+5n^4d^2+10n^3d^3+10n^2d^4+4nd^5+d^6.\] 
We claim that $g(n,d)\equiv 1\bmod 8$. If $n,d$ are both odd, then \[g(n,d)\equiv 1+2nd+5+10nd+10+4nd+1=17+16nd\equiv 1\bmod 8.\] If $n$ is even and $d$ is odd, then $g(n,d)\equiv d^6\equiv 1\bmod 8$. If $n$ is odd and $d$ is even, then 
$g(n,d)\equiv 1+2nd+5d^2\bmod 8$. Writing $n=2k+1$ for some integer $k$ we see that
    \[
    g(n,d)\equiv 5d^2+2d+1\equiv (d+1)^2\equiv 1\bmod 8,
    \]
which proves the claim. Letting $D=\sqf(g(n,d))$, the fact that $g(n,d)\equiv 1\pmod 8$ implies that $D\equiv 1\bmod 8$ and therefore 2 splits in $\bbQ(\sqrt D)=K$.

Similar reasoning shows that $g(n,d)$ is congruent to either 0 or 1 modulo 3. Considering all possible values of $n$ and $d$ modulo 9, we find that if $g(n,d)$ is divisible by 9, then $n$ and $d$ are both divisible by 3, which is a contradiction; hence $9\nmid g(n,d)$. Writing $g(n,d)=Ds^2$ for some integer $s$, this implies that $s$ is not divisible by 3, and therefore $g(n,d)\equiv D\bmod 3$. Hence, $D$ is congruent to 0 or 1 modulo 3, and therefore 3 is not inert in $K$.

A computation in \textsc{Magma} based on \cite[Thm. 2.1]{krumm_lgp} shows that $\pi_f$ has Dirichlet density $13/18$. Finally, Lemma \ref{unram} implies that every odd prime in $\pi_f$ is unramified in $K$, and we have already shown that 2 is unramified in $K$.
\end{proof}

An argument very similar to the proof of Theorem~\ref{10311_splitting} yields the following result, in which the relevant dynamical modular curve is known to be isomorphic to $X_1^\Ell(13)$.

\begin{thm}\label{1032_splitting} Let $K$ be a quadratic field such that ${G(f_c,K)\cong 10(3,2)}$ for some $c\in K$. Then the prime $2$ splits in $K$, and every prime in $\pi_f$ is unramified in $K$, where
\[f(x)=x^6 + 2x^5 +x^4 + 2x^3 + 6x^2 + 4x +1.\]
Moreover, the set $\pi_f$ has Dirichlet density $13/18$.
\end{thm}

Our next result concerns the curve $X_1^\Ell(16)$, which is isomorphic to the modular curve for the portrait $8(4)$; see Section 3.7 of \cite{doyle/faber/krumm:2014}. In contrast to Theorems \ref{10311_splitting} and \ref{1032_splitting}, we show that the discriminants of quadratic fields defined by points on $X_1^\Ell(16)$ are not restricted to any residue class. (Note that this proves Theorem \ref{main_class_group_thm}(a).)

\begin{thm}\label{84_discriminants}
Let $\calP$ denote the portrait $8(4)$. For every prime integer $p$ and every residue class $\frakc\in\bbZ/p\bbZ$, there exist infinitely many squarefree integers $d\in\frakc$ such that the curve $X_1(\calP)$ has a quadratic point defined over the field $\bbQ(\sqrt d)$.
\end{thm}

For the proof of the theorem we use the methods of \cite{krumm:2016sqfree} and \cite{krumm/pollack:2020}; the following lemma, which follows from Proposition 14 in \cite{krumm/pollack:2020}, collects the main tools to be used.

\begin{lem}\label{sqfree_lemma}
Let $f\in\bbZ[x]$ be a squarefree polynomial of degree at least $3$, and such that every irreducible factor of $f$ has degree at most $6$. Let
\[\calS(f)=\{\sqf(f(x)):x\in\bbQ\text{ and } f(x)\ne 0\}.\]
Let $D$ be the largest integer dividing all integer values of $f$. Fix a prime $p$ such that $f$ has an irreducible factor whose discriminant is not divisible by $p$, and let $\varepsilon=\varepsilon_p\in\{0,1\}$ be the parity of $\ord_p(D)$. Finally, for $\frakc\in\bbZ/p\bbZ$ and $v\in\bbZ$, let $\sigma(p,\frakc,v)$ denote the following statement.
\begin{equation}\label{sigma_pcv}
\sigma(p,\frakc,v):\;
\begin{cases}
\text{There exist $h\in\frakc$ and $x_0,y_0\in\bbZ$ satisfying}\\
\;\;\bullet\;\;hy_0^2\equiv f(x_0)\pmod*{p^{2(v+\varepsilon) + 1}}\text{ and}\\
\;\;\bullet\;\;\ord_p(y_0)=v+\varepsilon.
\end{cases}
\end{equation}
Suppose $\frakc$ is nonzero and $\sigma(p,\frakc,v)$ holds for some $v\ge 0$. Then the set $ \calS(f)\cap\frakc$ is infinite.
\end{lem}

\begin{proof}[Proof of Theorem \ref{84_discriminants}] By \cite[p. 93]{morton:1998}, the curve $X_1(\calP)$ is hyperelliptic and has an  affine model given by $y^2=f(x)$, where $f(x)=-x(x^2+1)(x^2-2x-1)$.

In order to prove the theorem it suffices to show that, for every prime $p$ and residue class $\frakc\in\bbZ/p\bbZ$, the set $\calS(f)\cap\frakc$ is infinite. Indeed, if $d$ belongs to this set, we may write $dy_0^2=f(x_0)$ for some rational numbers $x_0,y_0$ with $y_0\ne 0$. The pair $(x_0,y_0\sqrt{d})$ then represents a quadratic point on $X_1(\calP)$ whose field of definition is $\bbQ(\sqrt d)$. Hence, the theorem follows.

In the notation of Lemma \ref{sqfree_lemma}, for the above polynomial $f(x)$ we have $D=2$; moreover, the irreducible factor $x$ of $f(x)$ has discriminant $1$, which is coprime to every prime $p$. It follows in particular that
\[\varepsilon_p=
\begin{cases}
0 & \text{if $p$ is odd},\\
1 & \text{if $p=2$}.
\end{cases}\]

Fix a prime $p$. For the class $\frakc=0\in\bbZ/p\bbZ$, Theorem 2.1 in \cite{krumm:2016sqfree} implies that the set $\calS(f)\cap\frakc$ is infinite, as desired. (When $p=2$, the hypotheses of the cited theorem are not all satisfied, but the proof still applies.) Next, we claim that
\[\text{for every nonzero $\frakc\in\bbZ/p\bbZ$, either $\sigma(p,\frakc,0)$ or $\sigma(p,\frakc,1)$ must hold.}\]

Assuming this claim for the moment, Lemma \ref{sqfree_lemma} implies that the set $\calS(f)\cap\frakc$ is infinite, completing the proof of the theorem.

To prove the claim we consider first the case $p=2$: taking $\frakc=1$, the statement $\sigma(p,\frakc,1)$ can be shown to hold by setting $(h, x_0,y_0)=(1,16,4)$ in \eqref{sigma_pcv}. The remainder of the proof is divided into three cases.

\noindent\textbf{Case} $p\le 5$: Taking $p=3$, we check that $\sigma(p,\frakc,1)$ holds for $\frakc=1,2$ by using the tuples
\[(h,x_0,y_0)=(1,9,3)\quad\text{and}\quad\;(h,x_0,y_0)=(2,18,3).\]
Similarly, taking $p=5$, we check that $\sigma(p,\frakc,1)$ holds for $\frakc=1,2,3,4$, respectively, by using the following tuples $(h,x_0,y_0)$:
\[(1,25,5),\;(2,18,5),\;(3,75,5),\;(4,7,5).\]

In the remaining two cases we show that $\sigma(p,\frakc,0)$ holds. For $r\in\frakc$, let $X_r$ be the hyperelliptic curve over $\bbF_p$ defined by the equation $ry^2=f(x)$. Since $\varepsilon_p=0$, the statement $\sigma(p,\frakc,0)$ is equivalent to the requirement that $X_r$ have an affine point $(x_0,y_0)\in X_r(\bbF_p)$ with $y_0\ne 0$; we refer to such points as \emph{nontrivial} points on $X_r$. Thus, it remains to show that $X_r$ has at least one nontrivial point.

\noindent\textbf{Case} $7\le p\le 23$: A straightforward search for points verifies that
$\#X_r(\bbF_p)\ge 7$ for every nonzero $r\in\bbF_p$. (Note that it suffices to check this for just two values of $r$, one in each square class modulo $p$.) The number of affine points $(x_0,y_0)\in X_r(\bbF_p)$ having $y_0=0$ is at most $5$, so there must exist at least one nontrivial point in $X_r(\bbF_p)$, as required.

\noindent\textbf{Case} $p\ge 29$: For $r\in\bbF_p\smallsetminus 0$, the curve $X_r$ has genus 2, so the Hasse--Weil bound yields
\[\# X_r(\bbF_p)\ge \lfloor p+1-4\sqrt p\rfloor\ge 7.\]
The same reasoning as in the previous case implies that $X_r$ has a nontrivial $\bbF_p$-point.
\end{proof}

We end the paper by proving Theorem \ref{main_class_group_thm}(b).

\begin{prop}\label{class_group_84}
Let $\calP=8(4)$. There exist infinitely many imaginary quadratic fields $K$ with class number divisible by $10$, such that $X_1(\calP)$ has a quadratic point defined over $K$.
\end{prop}
\begin{proof}
As noted earlier, the curve $X_1(\calP)$ is isomorphic to $X_1^\Ell(16)$. The result follows from \cite[Cor. 3.2]{gillibert/levin:2012}, since the Jacobian $J_1(16)$ has a rational torsion point of order 10.
\end{proof}

\begin{rem}
Experimental evidence supports a statement stronger than Proposition \ref{class_group_84}: for \emph{every} imaginary quadratic field $K\ne\bbQ(\sqrt{-15})$ that is the field of definition of a point on $X_1(\calP)$, the class number of $K$ is divisible by 10. One approach to proving this is suggested by the methods of \cite{gillibert/levin:2012,gillibert/bilu:2018}; however, the required computational tools (in particular, for computing quotients of abelian varieties) do not seem to be presently available.
\end{rem}

\newpage
\appendix

\section{Dynamical modular curves of genera $1$ and $2$}\label{app:curve_models}

We provide here models for all dynamical modular curves $X_1(\calP)$ of genus $1$ or $2$, together with an explicit description of the morphism $c : X_1(\calP) \to \bbP^1$. Each of these models appears in \cite{poonen:1998}. Note that in some cases we provide two models---one of the form $y^2 = F(x)$, and another that turns out to be more useful for certain aspects of our proofs.

\begin{center}
	\begin{tabular}{|l|l|l|}
	\hline
	Portrait $\calP$ & Model(s) for $X_1(\calP)$ & Morphism $c : X_1(\calP) \to \bbP^1$\\\hline
		&& \\
	8(1,1)a & $y^2 = x^3 - x^2 + x$ & \small$ -\dfrac{(x^2 + 1)^2}{4x(x-1)^2}$\\
		&& \\
	8(1,1)b & $y^2 = 2(x^3 + x^2 - x + 1)$ & \small$ -\dfrac{2(x^2 + 1)}{(x+1)^2(x - 1)^2}$ \\
		&& \\
	8(2)a & $y^2 = x^3 - 2x + 1$ & \small$ -\dfrac{(x^2 - 2x + 2)(x^2 + 2x - 2)}{4x^2(x-1)}$\\
		&& \\
	8(2)b & $y^2 = 2(x^3 + x^2 - x + 1)$ & \small$ -\dfrac{x^4 + 2x^3 + 2x^2 - 2x + 1}{(x+1)^2(x - 1)^2}$ \\
		&& \\
	10(2,1,1)a & $y^2 = 5x^4 - 8x^3 + 6x^2 + 8x + 5$ & \small$ -\dfrac{(3x^2 + 1)(x^2 + 3)}{4(x+1)^2(x - 1)^2}$ \\
	
	    && \\
	& $y^2 + xy + y = x^3 - x^2 - x$ & \small$ \dfrac{x-2}{4x(x-1)}y - \dfrac{x^4 - x^3 + 3x - 1}{4x^2(x-1)}$\\
		&& \\
	10(2,1,1)b & $y^2 = (5x^2 - 1)(x^2 + 3)$ & \small$ -\dfrac{(3x^2 + 1)(x^2 + 3)}{4(x+1)^2(x - 1)^2}$ \\
	    && \\
	& $y^2 + xy + y = x^3 + x^2$ & \small$ -\dfrac{x+2}{4x(x+1)}y - \dfrac{x^4 + 4x^3 + 6x^2 + 3x + 1}{4x^2(x+1)}$\\
	    && \\\hline
	    &&\\
	8(3) & \small $y^2 = x^6 - 2x^4 + 2x^3 + 5x^2 + 2x + 1$ & \small $-\dfrac{x^6 + 2x^5 + 4x^4 + 8x^3 + 9x^2 + 4x + 1}{4x^2(x+1)^2}$ \\
		&& \\
	8(4) & \small $y^2 = -x(x^2 + 1)(x^2 - 2x - 1)$ & \small $ \dfrac{(x^2 - 4x - 1)(x^4 + x^3 + 2x^2 - x + 1)}{4x(x+1)^2(x-1)^2}$\\
		&& \\
	10(3,1,1) & \Small $y^2 = x^6 + 2x^5 + 5x^4 + 10x^3 + 10x^2 + 4x + 1$ & \small $ -\dfrac{x^6 + 2x^5 + 4x^4 + 8x^3 + 9x^2 + 4x + 1}{4x^2(x+1)^2}$\\
		&& \\
	10(3,2) & \small $y^2 = x^6 + 2x^5 + x^4 + 2x^3 + 6x^2 + 4x + 1$ & \small $-\dfrac{x^6 + 2x^5 + 4x^4 + 8x^3 + 9x^2 + 4x + 1}{4x^2(x+1)^2}$ \\
		&& \\\hline
	\end{tabular}
\end{center}

\vfill

\section{Tables of preperiodic portraits}\label{graphs_appendix}

\subsection{Preperiodic portraits realized over quadratic fields}
We list here the 46 portraits known to be realized as $G(f_c,K)$ for some quadratic field $K$ and $c \in K$. These were found in the search described in \cite{doyle/faber/krumm:2014}.
The label of each portrait is in the form $N(\ell_1,\ell_2,\ldots)$, where $N$ is the number of vertices in the portrait and $\ell_1,\ell_2,\ldots$ are the lengths of the directed cycles in the portrait in nonincreasing order. If more than one isomorphism class with this data was observed, we add a lowercase Roman letter to distinguish them. For example, the labels 5(1,1)a and 5(1,1)b correspond to the two isomorphism classes of portraits observed that have five vertices and two fixed points. In all figures, we omit the vertex corresponding to the fixed point at infinity.

% 10(2) and 12(2) are the only graphs where I used a scaling different from .6, and therefore had to use \includegraphics[scale=.58] instead of \pic.

\vspace{4mm}

\begin{center}

\begin{tabularx}{.98\textwidth}{|L|L|L|} \hline
0 & 2(1) & 3(1,1)
\end{tabularx} \offinterlineskip
\begin{tabularx}{.98\textwidth}{|H|H|H|}
 & \pic{graph2_1} & \pic{graph3_11} \\ \hline
\end{tabularx} \offinterlineskip
\begin{tabularx}{.98\textwidth}{|L|L|L|}
3(2) & 4(1) & 4(1,1)
\end{tabularx} \offinterlineskip
\begin{tabularx}{.98\textwidth}{|H|H|H|}
\pic{graph3_2} & \pic{graph4_1} & \pic{graph4_11} \\ \hline
\end{tabularx} \offinterlineskip
\begin{tabularx}{.98\textwidth}{|L|L|L|}
4(2) & 5(1,1)a & 5(1,1)b
\end{tabularx} \offinterlineskip
\begin{tabularx}{.98\textwidth}{|H|H|H|}
\pic{graph4_2} & \pic{graph5_11a} & \pic{graph5_11b} \\ \hline
\end{tabularx} \offinterlineskip
%
%\newpage
%
\begin{tabularx}{.98\textwidth}{|M|M|} 
5(2)a & 5(2)b
\end{tabularx} \offinterlineskip
\begin{tabularx}{.98\textwidth}{|W|W|}
\pic{graph5_2a} & \pic{graph5_2b} \\ \hline
\end{tabularx} \offinterlineskip
\begin{tabularx}{.98\textwidth}{|L|L|L|}
6(1,1) & 6(2) & 6(2,1)
\end{tabularx} \offinterlineskip
\begin{tabularx}{.98\textwidth}{|H|H|H|}
\pic{graph6_11} & \pic{graph6_2} & \pic{graph6_21} \\ \hline
\end{tabularx} \offinterlineskip
\begin{tabularx}{.98\textwidth}{|M|M|}
6(3) & 7(1,1)a
\end{tabularx} \offinterlineskip
\begin{tabularx}{.98\textwidth}{|W|W|}
\pic{graph6_3} & \pic{graph7_11a} \\ \hline
\end{tabularx} \offinterlineskip
\begin{tabularx}{.98\textwidth}{|M|M|}
7(1,1)b & 7(2,1,1)a
\end{tabularx} \offinterlineskip
\begin{tabularx}{.98\textwidth}{|W|W|}
\pic{graph7_11b} & \pic{graph7_211a} \\ \hline
\end{tabularx} \offinterlineskip
\newpage
\begin{tabularx}{.98\textwidth}{|M|M|}\hline
7(2,1,1)b & 8(1,1)a
\end{tabularx} \offinterlineskip
\begin{tabularx}{.98\textwidth}{|W|W|}
\pic{graph7_211b} & \pic{graph8_11a} \\ \hline
\end{tabularx} \offinterlineskip
\begin{tabularx}{.98\textwidth}{|M|M|}
8(1,1)b & 8(2)a
\end{tabularx} \offinterlineskip
\begin{tabularx}{.98\textwidth}{|W|W|}
\pic{graph8_11b} & \pic{graph8_2a} \\ \hline
\end{tabularx} \offinterlineskip
\begin{tabularx}{.98\textwidth}{|M|M|}
8(2)b & 8(2,1,1)
\end{tabularx} \offinterlineskip
\begin{tabularx}{.98\textwidth}{|W|W|}
\pic{graph8_2b} & \pic{graph8_211} \\ \hline
\end{tabularx} \offinterlineskip
\begin{tabularx}{.98\textwidth}{|M|M|}
8(3) & 8(4)
\end{tabularx} \offinterlineskip
\begin{tabularx}{.98\textwidth}{|W|W|}
\pic{graph8_3} & \pic{graph8_4} \\ \hline
\end{tabularx} \offinterlineskip
%
%\newpage
%
\begin{tabularx}{.98\textwidth}{|M|M|} %\hline
9(2,1,1) & 10(1,1)a
\end{tabularx} \offinterlineskip
\begin{tabularx}{.98\textwidth}{|W|W|}
\pic{graph9_211} & \pic{graph10_11a} \\ \hline
\end{tabularx} \offinterlineskip
\begin{tabularx}{.98\textwidth}{|M|M|}
10(1,1)b & 10(2)
\end{tabularx} \offinterlineskip
\begin{tabularx}{.98\textwidth}{|W|W|}
\pic{graph10_11b} & \includegraphics[scale=.58]{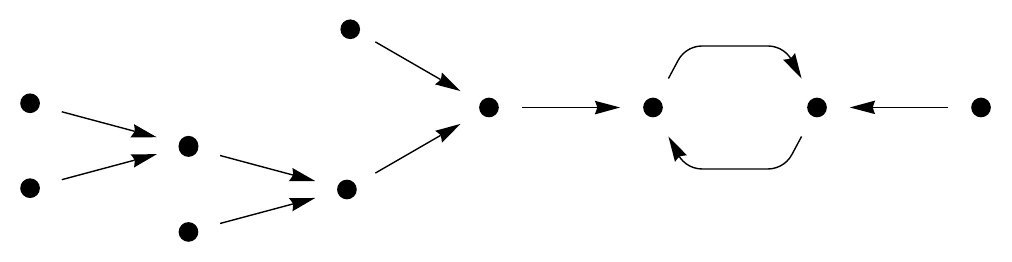} \\ \hline
\end{tabularx} \offinterlineskip
\begin{tabularx}{.98\textwidth}{|M|M|}
10(2,1,1)a & 10(2,1,1)b
\end{tabularx} \offinterlineskip
\begin{tabularx}{.98\textwidth}{|W|W|}
\pic{graph10_211a} & \pic{graph10_211b} \\ \hline
\end{tabularx} \offinterlineskip
\begin{tabularx}{.98\textwidth}{|L|L|L|}
10(3)a & 10(3)b & 10(3,1,1)
\end{tabularx} \offinterlineskip
\begin{tabularx}{.98\textwidth}{|H|H|H|}
\pic{graph10_3a} & \pic{graph10_3b} & \pic{graph10_311} \\ \hline
\end{tabularx} \offinterlineskip
\newpage
\begin{tabularx}{.98\textwidth}{|M|M|}\hline
10(3,2) & 12(2)
\end{tabularx} \offinterlineskip
\begin{tabularx}{.98\textwidth}{|W|W|}
\pic{graph10_32} & \includegraphics[scale=.58]{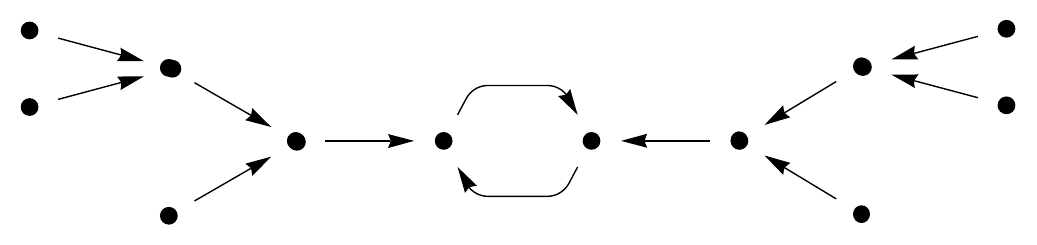} \\ \hline
\end{tabularx} \offinterlineskip
\begin{tabularx}{.98\textwidth}{|M|M|}
12(2,1,1)a & 12(2,1,1)b
\end{tabularx} \offinterlineskip
\begin{tabularx}{.98\textwidth}{|W|W|}
\pic{graph12_211a} & \pic{graph12_211b} \\ \hline
\end{tabularx} \offinterlineskip
%
%\newpage
%
\begin{tabularx}{.98\textwidth}{|L|L|L|} %\hline
12(3) & 12(4) & 12(4,2)
\end{tabularx} \offinterlineskip
\begin{tabularx}{.98\textwidth}{|H|H|H|}
\pic{graph12_3} & \pic{graph12_4} & \pic{graph12_42} \\ \hline
\end{tabularx} \offinterlineskip
\begin{tabularx}{.98\textwidth}{|M|M|}
12(6) & 14(2,1,1)
\end{tabularx} \offinterlineskip
\begin{tabularx}{.98\textwidth}{|W|W|}
\pic{graph12_6} & \pic{graph14_211} \\ \hline
\end{tabularx} \offinterlineskip
\begin{tabularx}{.98\textwidth}{|M|M|}
14(3,1,1) & 14(3,2)
\end{tabularx} \offinterlineskip
\begin{tabularx}{.98\textwidth}{|W|W|}
\pic{graph14_311} & \pic{graph14_32} \\ \hline
\end{tabularx} \offinterlineskip
\end{center}

\newpage
\subsection{Additional portraits}

The following portraits are not known to be realized over quadratic fields (and, in some cases, have been shown not to be); however, they make an appearance in the discussion in Section~\ref{sec:generic_properly_containing}, so we include them here. The labels $G_1,\ldots,G_{10}$ are taken from \cite{doyle:2018quad}.

\vspace{4mm}

\begin{center}
\begin{tabularx}{.98\textwidth}{|M|M|}
\hline
{$\boldsymbol{G_1}$} & {$\boldsymbol{G_2}$}
\end{tabularx} \offinterlineskip
\begin{tabularx}{.98\textwidth}{|W|W|}
\includegraphics[scale=.5]{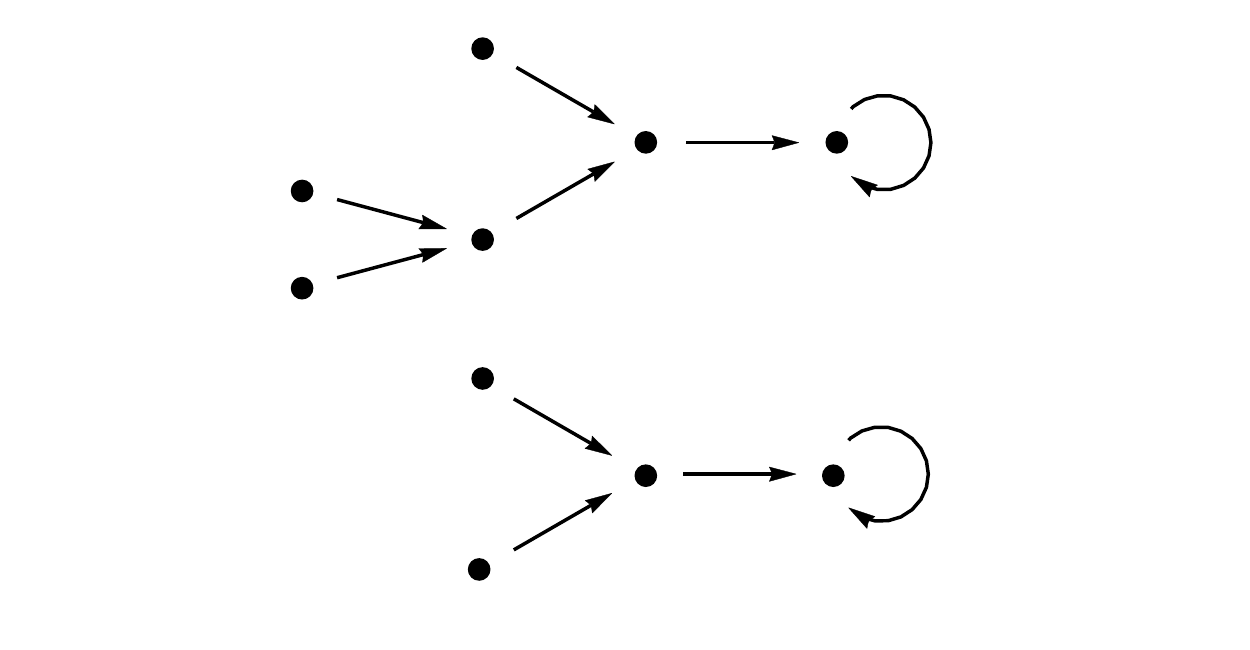} & \includegraphics[scale=.5]{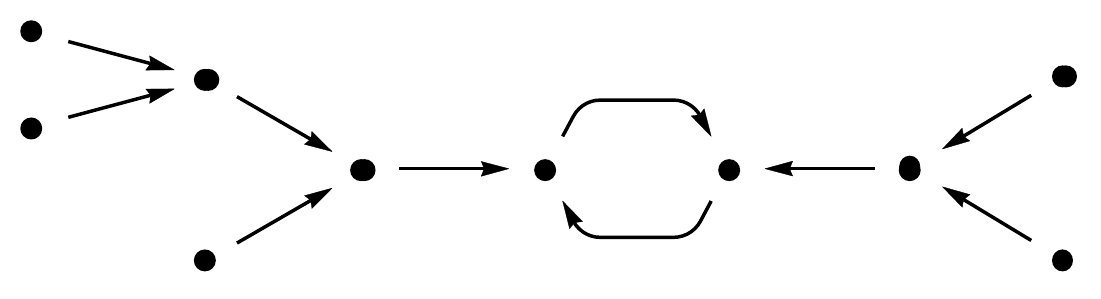}\\ \hline
\end{tabularx} \offinterlineskip
\begin{tabularx}{.98\textwidth}{|M|M|}
{$\boldsymbol{G_3}$} & {$\boldsymbol{G_4}$}
\end{tabularx} \offinterlineskip
\begin{tabularx}{.98\textwidth}{|W|W|}
\includegraphics[scale=.45]{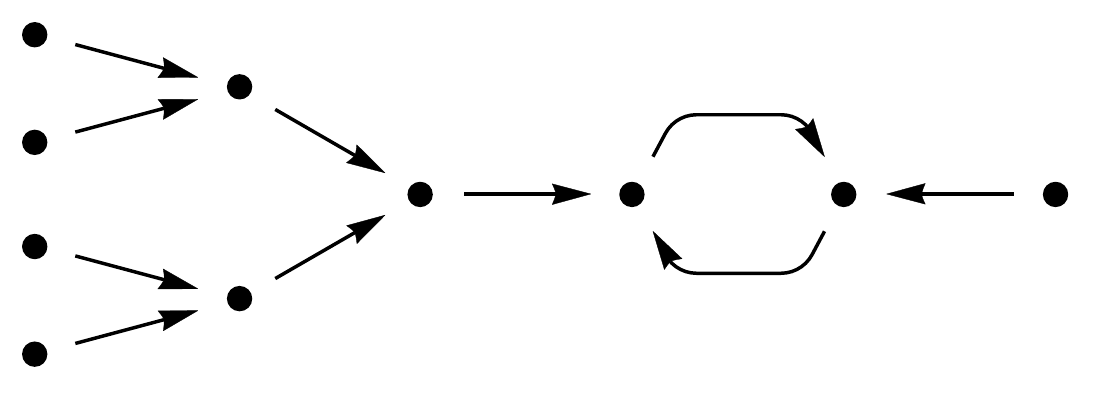} & \includegraphics[scale=.45]{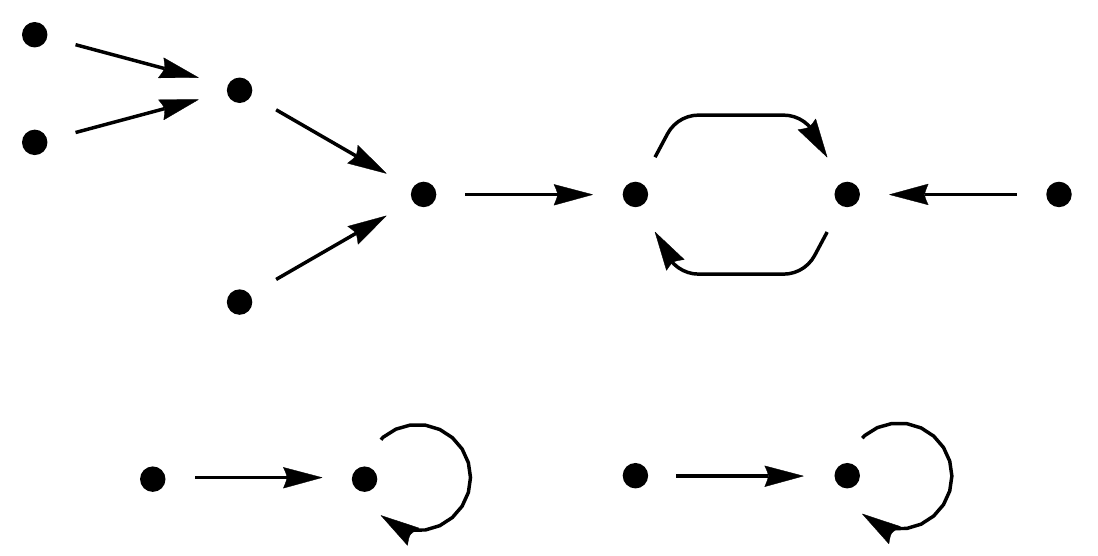}\\ \hline
\end{tabularx} \offinterlineskip
\begin{tabularx}{.98\textwidth}{|M|M|}
{$\boldsymbol{G_5}$} & {$\boldsymbol{G_6}$}
\end{tabularx} \offinterlineskip
\begin{tabularx}{.98\textwidth}{|W|W|}
\includegraphics[scale=.43]{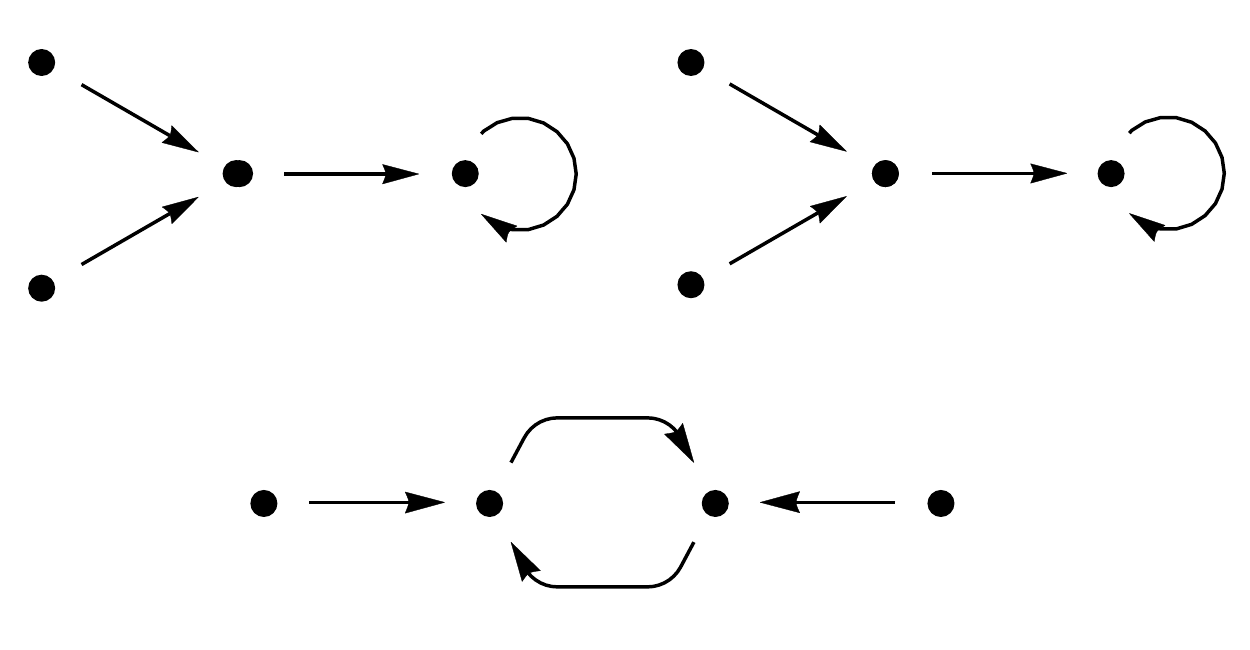} & \includegraphics[scale=.45]{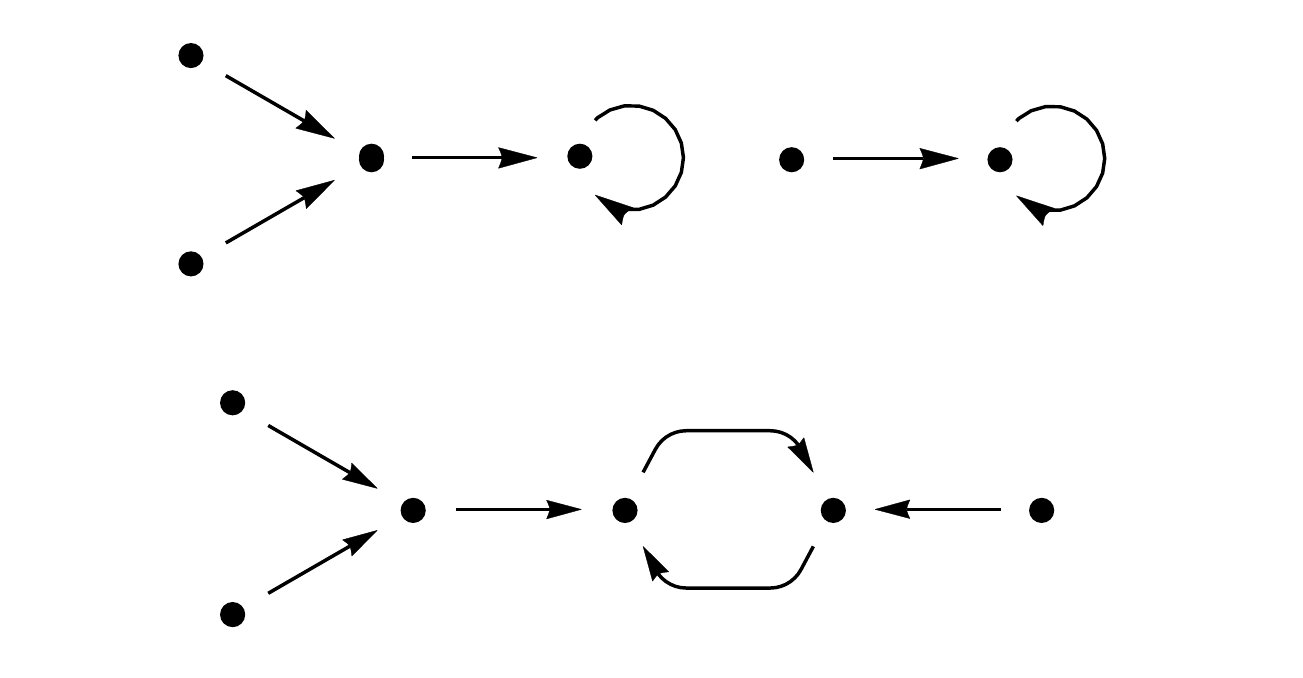}\\ \hline
\end{tabularx} \offinterlineskip 
\begin{tabularx}{.98\textwidth}{|M|M|}
{$\boldsymbol{G_7}$} & {$\boldsymbol{G_8}$}
\end{tabularx} \offinterlineskip
\begin{tabularx}{.98\textwidth}{|W|W|}
\includegraphics[scale=.53]{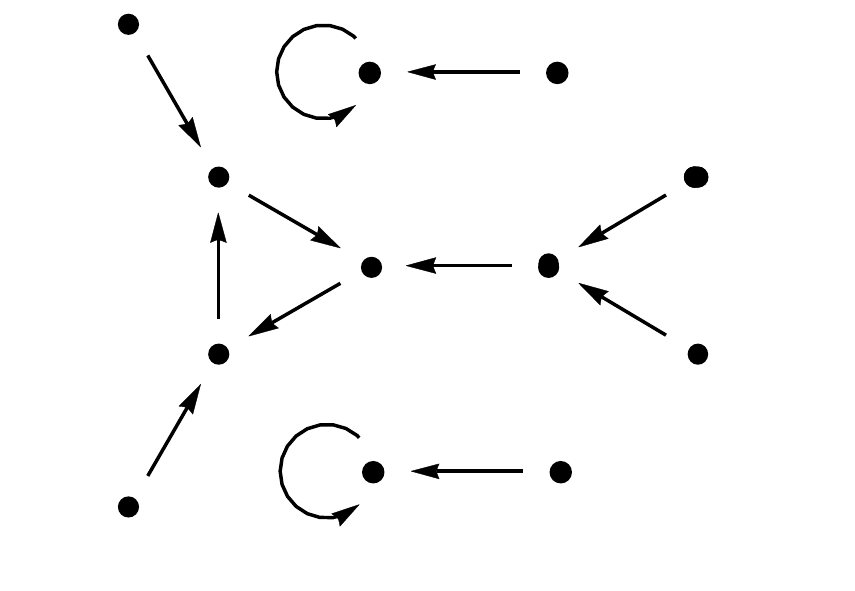} & \includegraphics[scale=.45]{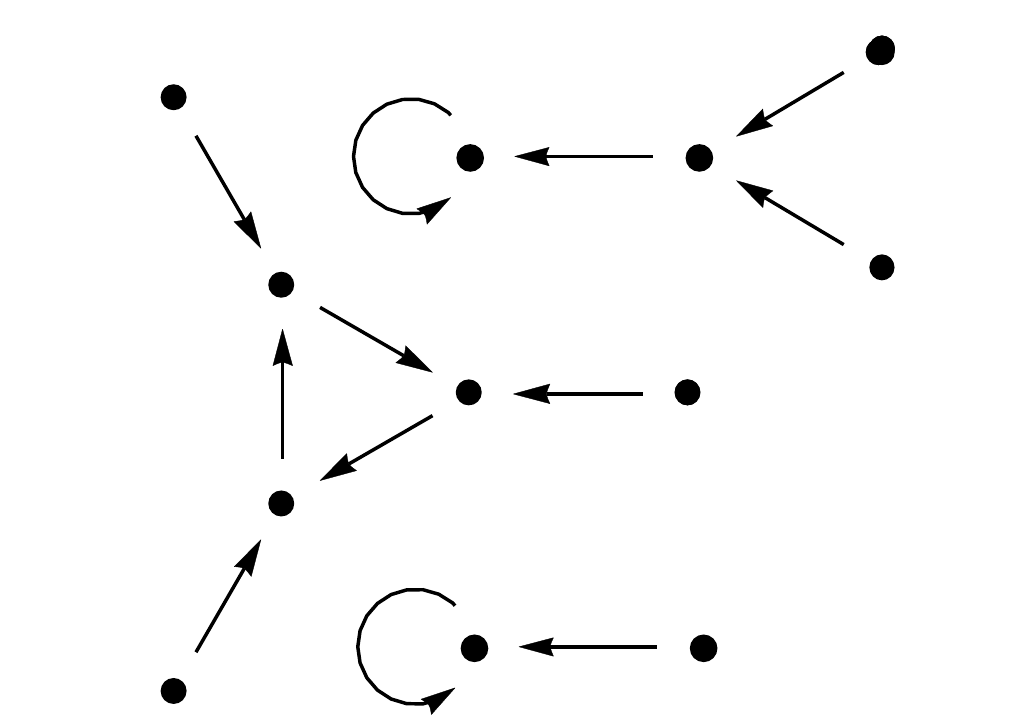}\\ \hline
\end{tabularx} \offinterlineskip
\begin{tabularx}{.98\textwidth}{|M|M|}
{$\boldsymbol{G_9}$} & {$\boldsymbol{G_{10}}$}
\end{tabularx} \offinterlineskip
\begin{tabularx}{.98\textwidth}{|W|W|}
\includegraphics[scale=.53]{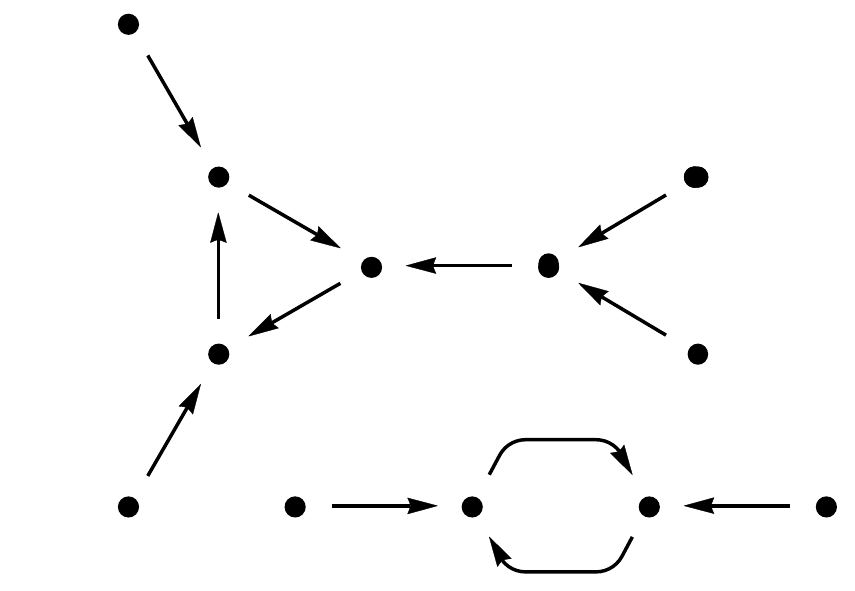} & \includegraphics[scale=.45]{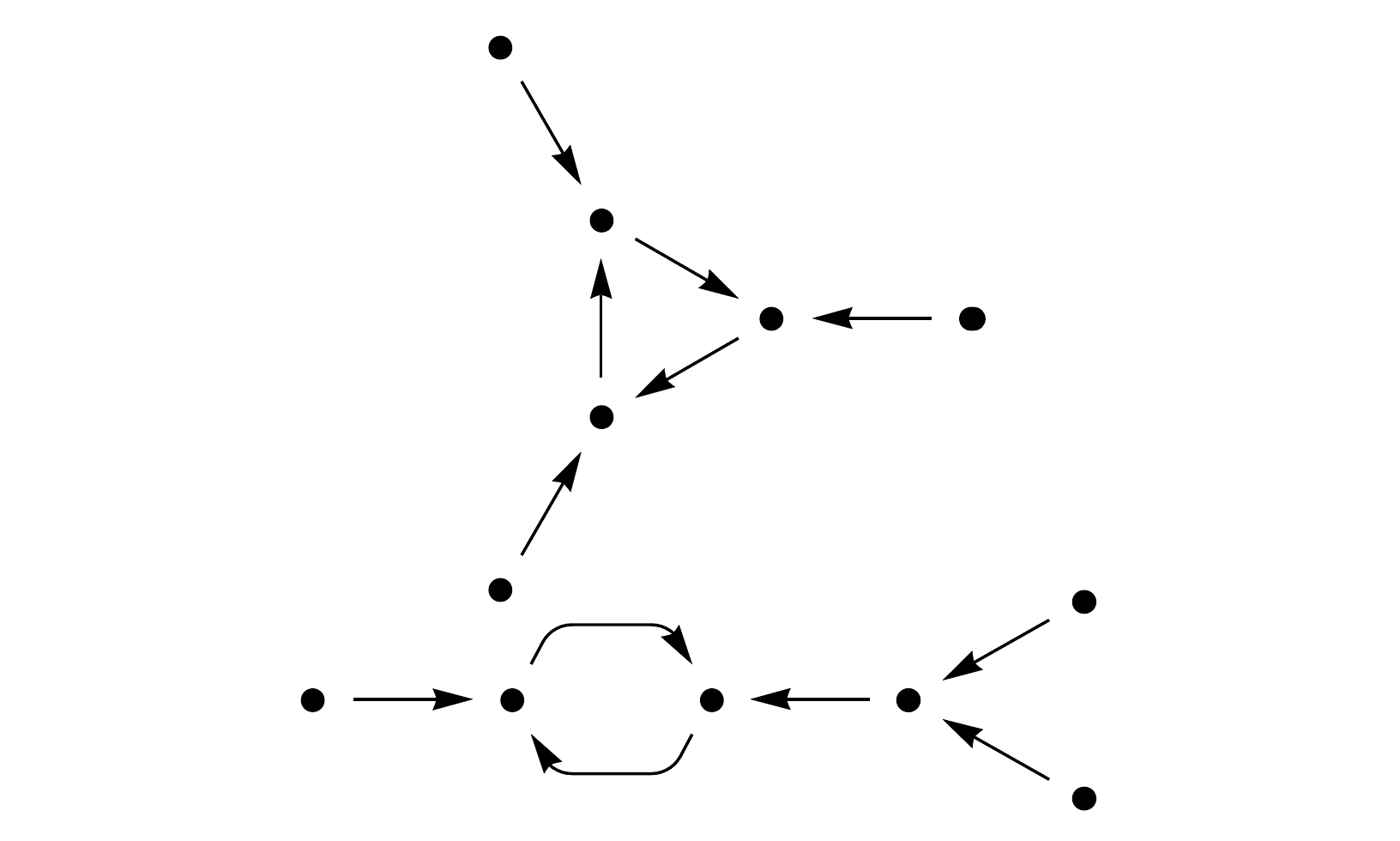}\\ \hline
\end{tabularx} \offinterlineskip
\end{center}

\vfill
{}

\bibliography{references}

\bibliographystyle{amsplain}
\end{document}